\documentclass[12pt,a4paper]{article}
\usepackage[english]{babel}
\usepackage{amsfonts}
\usepackage{authblk}
\usepackage{amscd}
\usepackage{latexsym}
\usepackage{amsthm}
\usepackage{amsmath}
\usepackage{graphicx}
\usepackage{graphics}
\usepackage{multirow}
\usepackage{fancyvrb}
\usepackage{color}
\usepackage{verbatim}
\usepackage{hyperref}
\usepackage{subfigure}
\usepackage{pgfplots}
\usepackage{enumerate}

\hypersetup{colorlinks=false}
\tikzset{ font={\fontsize{9pt}{12}\selectfont}}
\newtheorem{theorem}{Theorem}[section]
\newtheorem{proposition}[theorem]{Proposition}
\newtheorem{lemma}[theorem]{Lemma}

\newtheorem{corollary}[theorem]{Corollary}
\theoremstyle{definition}

\newtheorem{remark}[theorem]{Remark}

\newcommand{\com}{\mathbb{C}}
\newcommand{\wcom}{\widehat{\mathbb{C}}}

\newcommand{\nat}{\mathbb{N}}
\newcommand{\rat}{\mathbb{Q}}
\newcommand{\dis}{\mathbb{D}}

\begin{document}

\title{Connectivity of the Julia set for the Chebyshev-Halley family on degree $n$ polynomials}
\author{\footnote{Instituto de Matem\'aticas y Aplicaciones de Castell\'on (IMAC), Universitat Jaume I. Spain} B. Campos, \footnote{Universit\'e Paris-Est Marne-la-Vall\'ee. France} J. Canela and $^*$P. Vindel \\
{\footnotesize campos@uji.es, canela@maia.ub.es, vindel@uji.es}}
\date{}
\maketitle

\begin{abstract}
We study the Chebyshev-Halley family of root finding algorithms from the point of view of holomorphic dynamics.
 Numerical experiments show that the speed of convergence to the roots may be slower when the basins of attraction are not simply connected.
  In this paper we provide a criterion which guarantees the simple connectivity of the basins of attraction of the roots. 
We use the criterion for the Chebyshev-Halley methods applied to the degree $n$ polynomials $z^{n}+c$, obtaining a characterization of the parameters for which all Fatou components are simply connected and, therefore, the Julia set is connected. We also study how increasing $n$ affects the dynamics.
\end{abstract}

\section{Introduction}

Most of the problems faced by scientists and engineers involve equations that do not have a known
analytical solution.
 Numerical methods are  a good option to tackle and
solve real world problems. In particular,
iterative methods are used to find  approximations of the solutions of  $f(z)=0$.

The best-known  root-finding algorithm is Newton's method, which has  order of convergence 2. Many numerical  methods of order three or more are derived from Newton's scheme:
 Chebyshev method, also known as  super-Newton method (see \cite{Kneisl}, for example),
 Halley's method
and  super-Halley method.
A more detailed study of the construction and evolution of these numerical methods can be seen in \cite{olivo}.
These methods belong to a family of numerical algorithms called the Chebyshev-Halley family, which is given by
\begin{equation}
 x_{n+1} =x_{n}-\left( 1+\frac{1}{2}\;\frac{L_{f}\left(  x_{n}\right) }
 {1-\alpha L_{f}\left(  x_{n}\right) }\right) \frac{f\left(  x_{n}\right) }{f^{\prime
}\left(  x_{n}\right) }, \label{metodo}
\end{equation}
where
\begin{equation*}
L_{f}\left( x_{n}\right) =\frac{f\left( x_{n}\right) f^{\prime \prime }\left(
x_{n}\right) }{\left( f^{\prime }\left( x_{n}\right) \right) ^{2}}
\end{equation*}
and $\alpha\in\mathbb{C}$. Within this family, Chebyshev method is obtained  for $\alpha=0$, Halley's method is obtained for $\alpha=\frac{1}{2}$
and super-Halley method is obtained for $\alpha=1$.  Moreover, as $\alpha$ tends to $\infty$ these algorithms converge to Newton's method.

 Failures in intermediate calculations made by a computer are very difficult to detect. Consequently, one of the common aims of
numerical analysis is to select robust algorithms, that is, algorithms with
a good numerical stability in a wide range of situations. Therefore, it is
usual to consider iterative methods with high order of convergence. However, the radii of convergence which ensure
that the solution of the method is correct may decrease when we increase the order of the method.
One way to address this issue is to study the numerical method from
a dynamical point of view, i.e., to consider the iterative method as a
discrete dynamical system and to study its stability. This is a line of work
that has proven to be especially fruitful in recent years (see, for example, the
papers  \cite{Amat},  \cite{ArgyrosAlberto2}, \cite{ArgyrosAlberto},
\cite{alfac}, \cite{familiac}, \cite{king},  \cite{hernandez}).

In this paper we  study the  Chebyshev-Halley  family from a dynamical point of view in order to find which members of the family have better stability.
To carry out this dynamical study,  the root-finding algorithm is applied to a polynomial $P$.
By doing so, we obtain a rational map $Q:\wcom\rightarrow\wcom$, where $\wcom$ denotes de Riemann sphere, whose dynamics describes how behaves the method  when applied to the polynomial $P$.
Indeed, the points which converge to the roots of $P$ when applying to the numerical method are exactly those which converge to the roots of $P$ when iterating the map $Q$.

We shall give a short introduction to the concepts used in holomorphic dynamics. A more detailed  description of the topic can be found  in \cite{beardon} and \cite{Milnor}.
We consider the discrete dynamical system given by the iterates of
a rational map $Q:\widehat{\mathbb{C}}\rightarrow \widehat{\mathbb{C}}$.
A point $z_{0}\in \widehat{\mathbb{C}}$ is called \emph{fixed} if
$Q\left(z_{0}\right) =z_{0}$ and periodic of period
$p>1$ if $Q^{p}\left( z_{0}\right) =z_{0}$ and $Q^{k}\left( z_{0}\right) \neq
z_{0}$ for $k<p$. In the later case, we say that $\langle z_0 \rangle=\{z_0, Q(z_0), \cdots, Q^{p-1}(z_0)\}$ is a $p$-cycle.  A point is \emph{preperiodic} if it is eventually mapped under iteration of $Q$ onto a periodic point. The multiplier of a fixed point $z_0$ is given by  $\lambda
=Q^{\prime }(z_{0})$. Analogously, the multiplier of a $p$-cycle $\langle z_0\rangle$ is given by $\left(Q^p\right)'(z_0)$.  We say that a fixed point or a cycle is \emph{attracting} if
$|\lambda |<1$  (\emph{superattracting} if $\lambda =0$),  \emph{repelling} if
$|\lambda |>1$, and \emph{indifferent} if $|\lambda |=1$. In the later case $\lambda =e^{2\pi i\theta }$, where $\theta \in [ 0,1)$. We say that an indifferent point or cycle is \emph{rationally
indifferent} or \emph{parabolic} if $\theta \in\rat$.
Any attracting or parabolic point $z_0$  has a \emph{basin of attraction}, an open set of points which converge under iteration of $Q$ to $z_0$, related to it. We denote it by
\begin{equation*}
\mathcal{A}\left( z_{0}\right) =\{z\in \widehat{\mathbb{C}}\ :\ Q^{k}\left(
z\right) {\rightarrow }z_{0}\ \ \text{when}\ \ k{\rightarrow }\infty \}.
\end{equation*}
The basin of attraction of an attracting (or parabolic) $p$-cycle can be defined analogously using that all elements of $\langle z_0 \rangle$ are attracting (or parabolic) fixed points of $Q^p$.

 The dynamics of rational map $Q$ splits the Riemann sphere into two totally invariant subsets.  The \emph{Fatou set}, $\mathcal{F}\left( Q\right) $,  consists of the  $z\in \widehat{\mathbb{C}}$ such that the
family of iterates of $Q$, $\{Q(z),Q^{2}(z),\ldots ,Q^{k}(z),\ldots \}$,
is normal, or equivalently equicontinuous, in some open neighbourhood $U$ of $z$. The Fatou set is open and corresponds to the set of points with stable dynamics. Its complement, the
\emph{Julia set} $\mathcal{J}(Q)$, is  closed and corresponds to the set of points which present chaotic behaviour.
The connected components of the Fatou set, called \emph{Fatou components},
are mapped among themselves under iteration. All Fatou components of a rational map are either periodic or
preperiodic (\cite{Su}). By the Classification Theorem (see e.g.\ \cite{Milnor}), all periodic Fatou components are either basins of attraction of
attracting or parabolic cycles, or simply connected rotation domains (Siegel disks) or doubly connected rotation domains (Herman rings). Moreover, all periodic Fatou components are related to a critical point, that is, a point $z\in\wcom$ such that $Q'(z)=0$. Indeed,  the basin of attraction of an attracting or a parabolic cycle contains, at least, a critical point. Also, the orbit of,
at least, a critical point accumulates on  the
boundary of a Siegel disk or a Herman ring.

When a root-finding algorithm is studied from the point of view of holomorphic
dynamics, it is usual to apply the method to low degree polynomials (see for instance \cite{gato}, \cite{alfac}, \cite{familiac}, \cite{king}).
The reason is that when the degree of the polynomial increases,  the number of critical points of the rational map obtained  also increases.
 This is a serious drawback when analysing the
parameter spaces of the methods applied to high degree polynomials, since the orbits of the critical points are crucial to establish the existence of basins of attraction which do not come from the roots.

In the paper \cite{jordi}, we consider the Chebyshev-Halley family of numerical methods applied to the degree $n$ polynomials $z^n+c$, where $z\in\com$ and $c\in\com\setminus\{0\}$. As far as we know, this is the first time that a study of a family of root-finding algorithms applied to degree $n$ polynomials is considered from the point of view of dynamical systems and the corresponding parameter spaces are provided. In Figure \ref{fig:paramgaton} we show the parameter spaces of this family for different values of $n$.
Despite of the increase in the number of critical points, we use the symmetries of the dynamical system obtained to justify that it is enough to follow the orbit of a single critical point to determine the existence of basins of
attraction other than the ones provided by the roots.
 This property allows us to obtain pictures of the parameter spaces for polynomials of this family.
  We use these numerical studies together with some theoretical results to provide a first analysis of how the sets of parameters with good dynamical behaviour vary as $n$ increases.
  For fixed $n$, we show that  there exists a unique parameter $\alpha$ for which the Chebyshev-Halley method has order of convergence 4.
 However, the dynamics of such points may vary when we increase $n$. For instance, for $\alpha=1$ (super-Halley method) the algorithm has order of convergence 4 for $n=2$, but it presents bad dynamical behaviour for big $n$ (see Figures \ref{fig:dynam3}, \ref{fig:dynam10} and \ref{fig:dynam25}).

The goal of this paper is to continue the research began in \cite{jordi}. We focus on the study of pathological dynamical behaviour which appears both in the dynamical and the parameter planes.
 First, we study the dynamical conditions which lead to non-simply connected basins of attraction of the roots. It has been numerically observed that the basins of attraction may have holes, which seem to lead to slower speed of convergence.
 In Section~\ref{sec_Conectividad} we provide a dynamical condition (Proposition~\ref{propconn}) which can be applied to any holomorphic family of rational maps with a single free critical point (modulo symmetries).
  Using this condition, we prove that the basins of attractions of the roots for the Chebyshev-Halley family are multiply connected  if and only if the immediate
basin of attraction of the root $z=1$ contains another critical point  $c\neq 1$, and no preimage of $z=1$ other than himself. This characterization is used to study the connectivity of their Julia sets.  In Theorem \ref{thmconJulia} we prove that the  Julia set of the Chebyshev-Halley operator is disconnected if and only if the previous condition holds.  These results allow us to locate  the values of the parameter for which the Julia set is disconnected,  and
therefore, the numerical methods are more unstable.

\begin{figure}[p]
\centering
\subfigure[\small{n=2} ]{
    \begin{tikzpicture}
    \begin{axis}[width=210pt, axis equal image, scale only axis,  enlargelimits=false, axis on top]
      \addplot graphics[xmin=-1.4,xmax=4.6,ymin=-2,ymax=2] {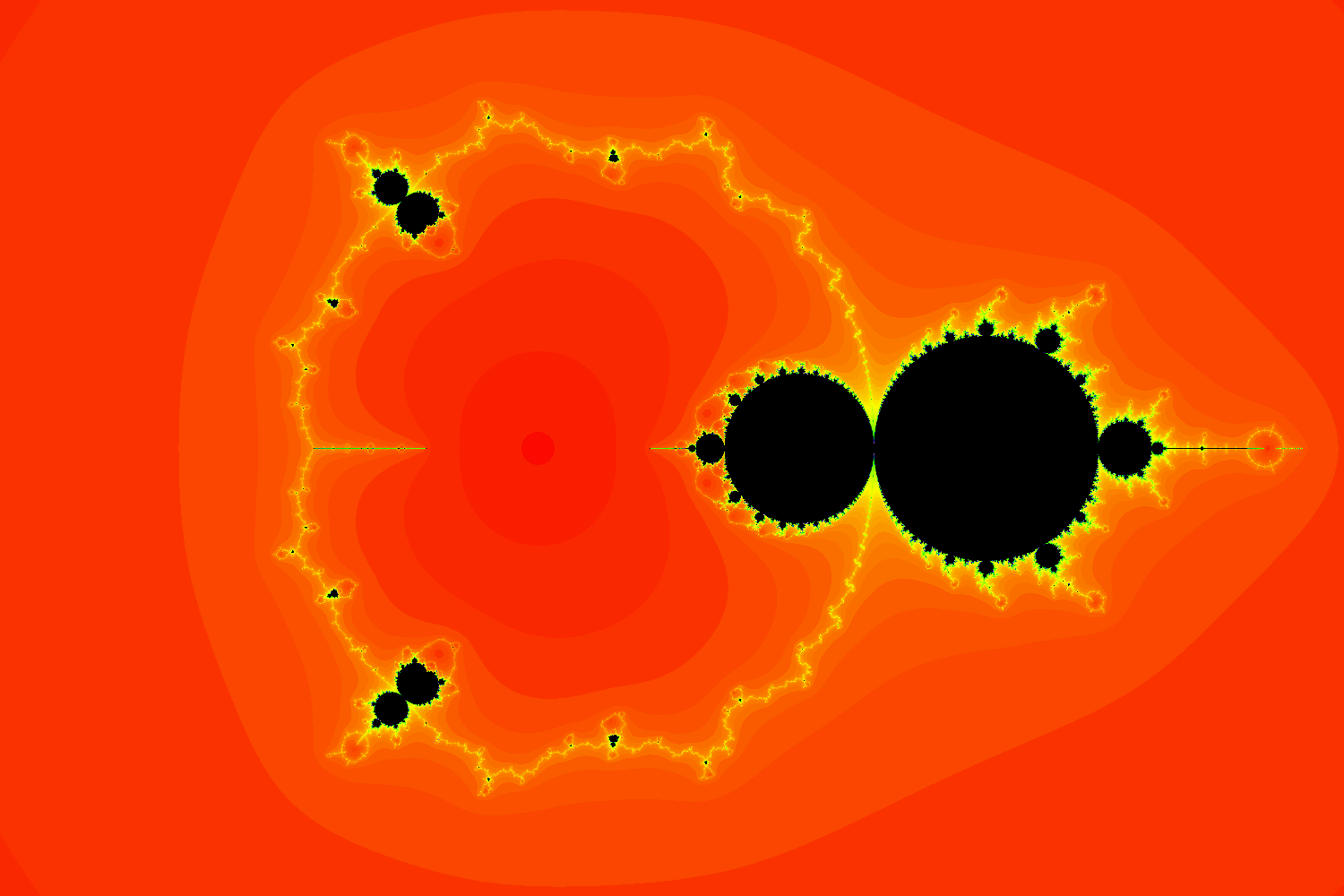};
    \end{axis}
  \end{tikzpicture}
    }
    \hfill
  \subfigure[\small{n=3} ]{
    \begin{tikzpicture}
    \begin{axis}[width=210pt, axis equal image, scale only axis,  enlargelimits=false, axis on top]
      \addplot graphics[xmin=-1.4,xmax=4.6,ymin=-2,ymax=2] {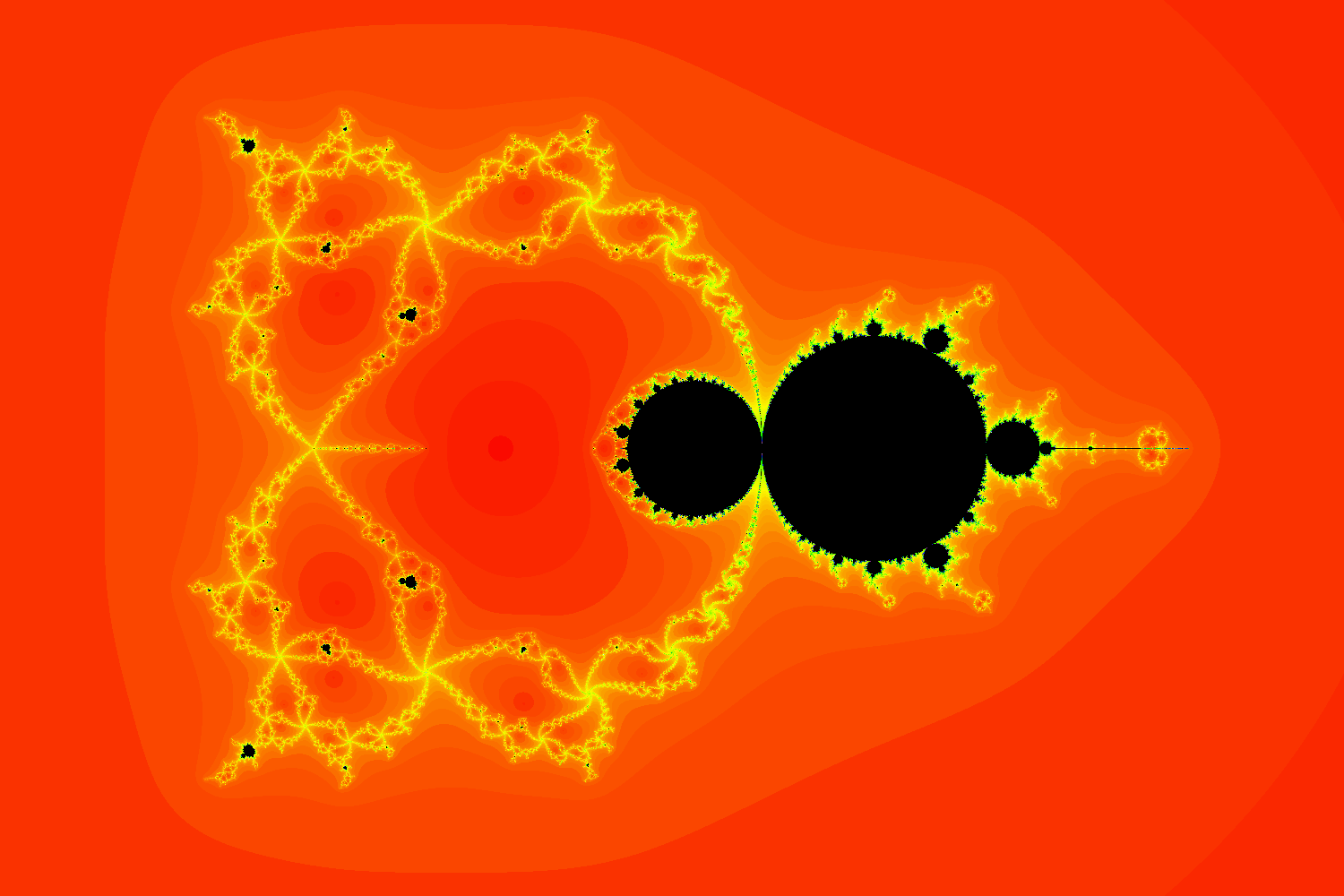};
    \end{axis}
  \end{tikzpicture}
    }

    \subfigure[\small{n=5}]{
	   	
    	\begin{tikzpicture}
    		\begin{axis}[width=210pt, axis equal image, scale only axis,  enlargelimits=false, axis on top]
      			\addplot graphics[xmin=-1.4,xmax=4.6,ymin=-2,ymax=2] {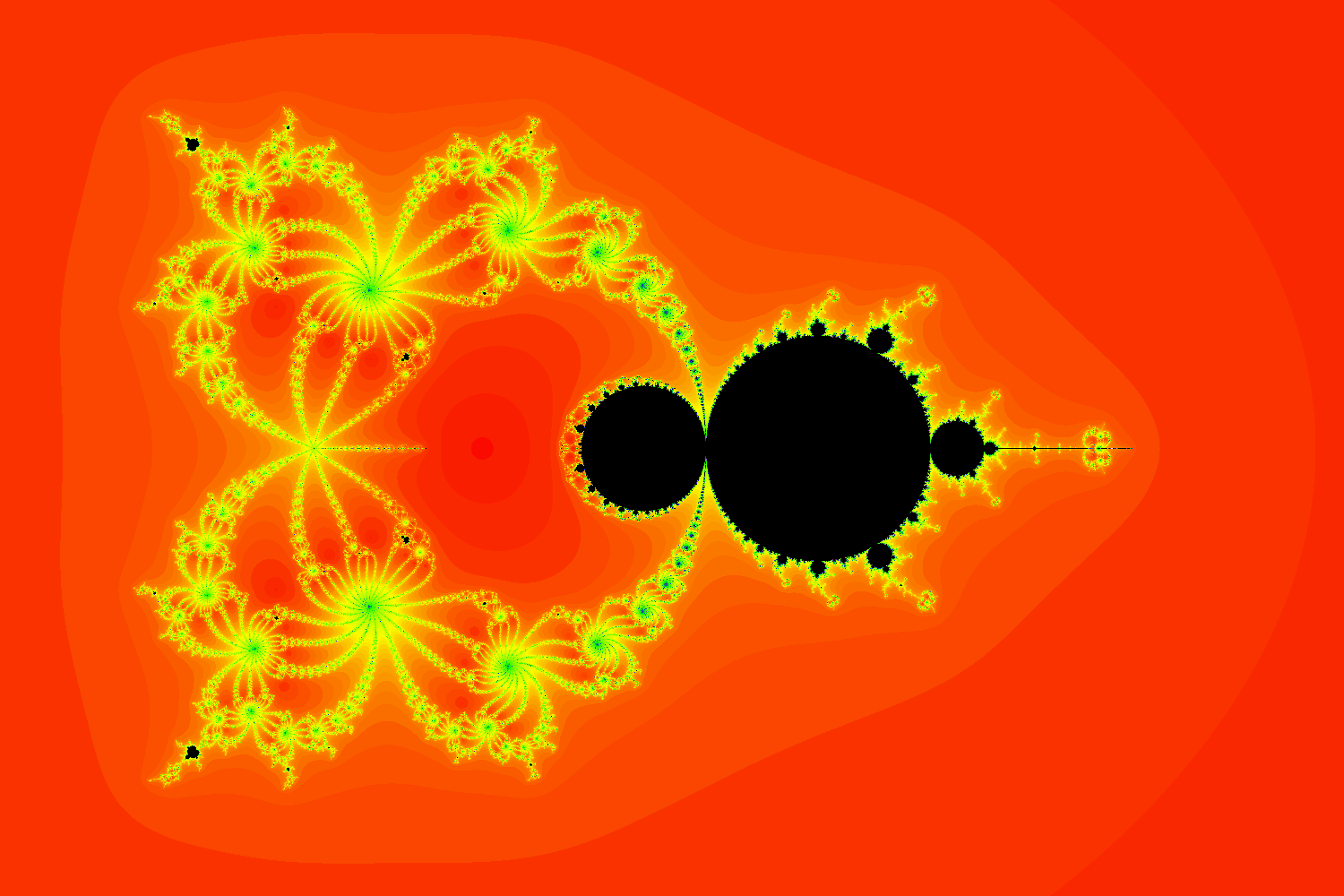};
    		\end{axis}
  		\end{tikzpicture}	
  }
  \subfigure[\small{n=10} ]{
    \begin{tikzpicture}
    \begin{axis}[width=210pt, axis equal image, scale only axis,  enlargelimits=false, axis on top]
      \addplot graphics[xmin=-1.4,xmax=4.6,ymin=-2,ymax=2] {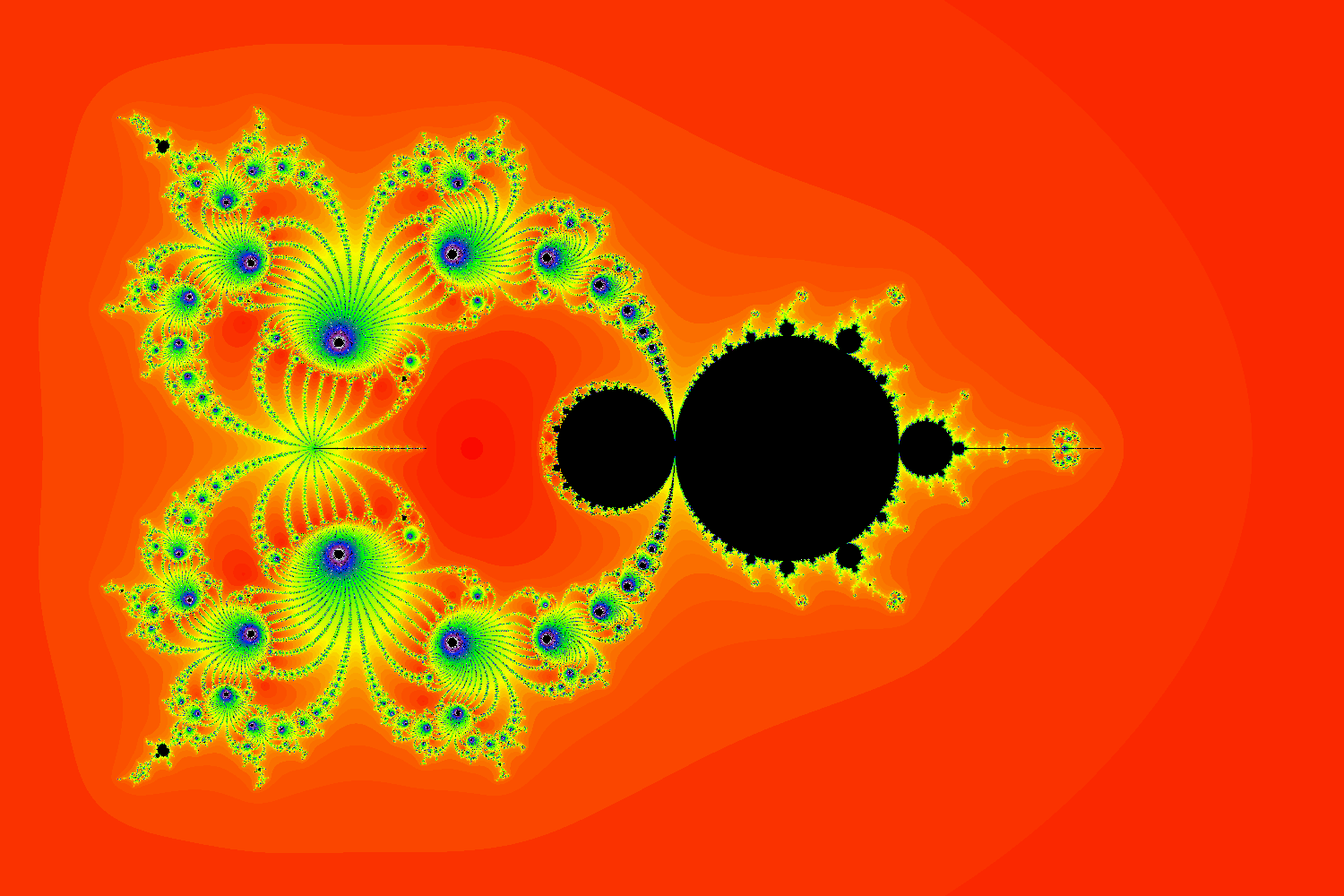};
    \end{axis}
  \end{tikzpicture}
    }

    \subfigure[\small{n=25}]{
    	\begin{tikzpicture}
    		\begin{axis}[width=210pt, axis equal image, scale only axis,  enlargelimits=false, axis on top]
      			\addplot graphics[xmin=-1.4,xmax=4.6,ymin=-2,ymax=2] {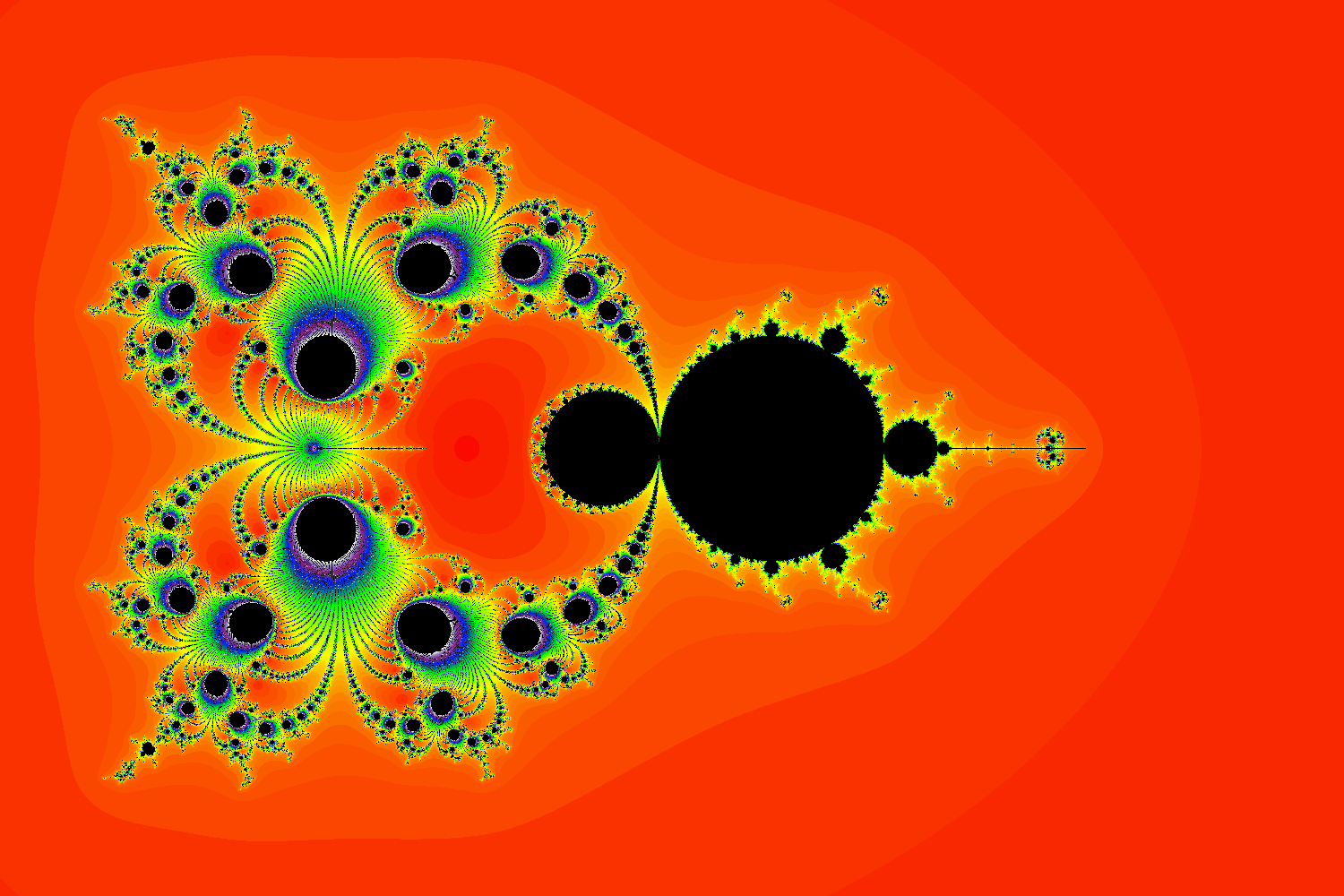};
    		\end{axis}
  		\end{tikzpicture}	
  }
  \subfigure[\small{n=100} ]{
    \begin{tikzpicture}
    \begin{axis}[width=210pt, axis equal image, scale only axis,  enlargelimits=false, axis on top]
      \addplot graphics[xmin=-1.4,xmax=4.6,ymin=-2,ymax=2] {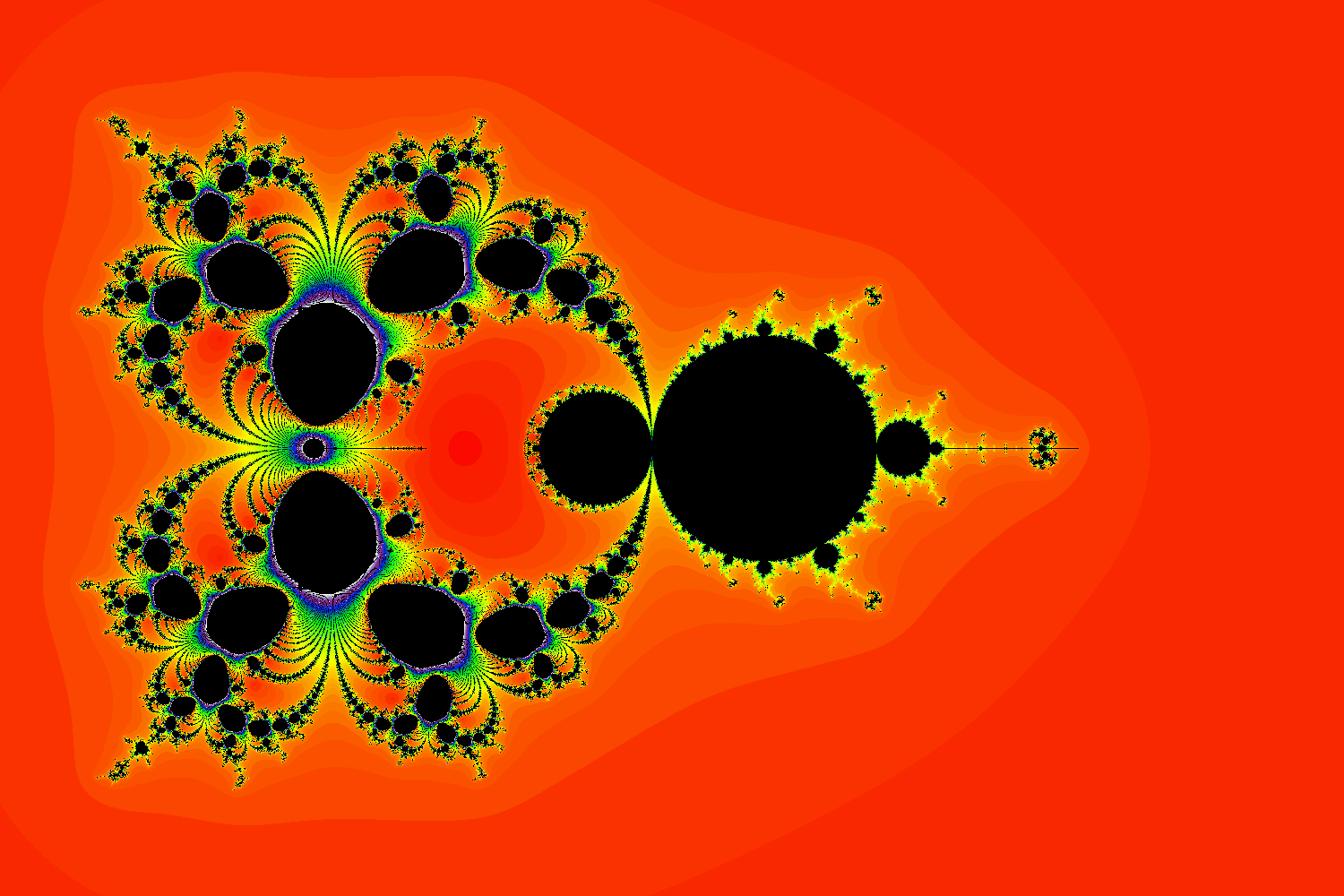};
    \end{axis}
  \end{tikzpicture}
    }
      \caption{\small Parameter spaces of the Chebyshev-Halley family applied to $z^n-1$.% for  $n\in\{2,3,5,10,25,100\}$.% The parameters $\alpha$ are taken so that $\re(\alpha)\in(-1.4, 4.6)$ and $\im(\alpha)\in(-2,2)$.
      \label{fig:paramgaton}}
    \end{figure}

Afterwards, we study the parameters which present bad behaviour when drawing the parameter planes. Parameter planes of the Chebyshev-Halley family applied to $z^n-1$ are shown in Figure~\ref{fig:paramgaton}, for several values of $n$. The drawings are obtained as follows.
 For each parameter $\alpha$ we iterate a critical point up to 150 times.
  If the orbit of the critical point converges to a root (that is, the iterate $w$ satisfies $|w-\xi|<10^{-4}$ for some $n$th-root of the unity $\xi$) in  less than 150 iterates, we plot the corresponding point using a scale from red (fast convergence) to green to purple and to grey (slow convergence).
   If after 150 iterates the critical orbit has not converged to a root, we plot the point in black. See Section \ref{sec_Numerico_param} for a more detailed explanation on how the images of parameter planes are produced. We observe in red parameters for which the corresponding critical points converge fast to the roots.
    We can also observe  the \textit{Cat set}, the set of  parameters for which the critical
 points do not belong to the basins of attraction of the roots. In all the figures we can distinguish two big disks called the head an the body of the Cat. These sets correspond to parameters for which strange fixed points are attracting (see Proposition~\ref{caracterinfty} and Proposition~\ref{caracterExt}).
  The Cat set consists of the head and the body (with their decorations) and a necklace-like structure which surrounds the head, that we call the \textit{Collar}.
  For $n$ small,  we observe  the Collar coloured in yellow. However, for $n\in\{25,100\}$ there appear some black disks which do not correspond to stable behaviour, as is the case for the head and the body. In this paper we analyse around which bifurcation parameters these regions appear (see Section~\ref{sec_bifurcation} and \ref{sec_Numerico_param}), and study the associated  operators from a numerical point of view (Section~\ref{sec_Numerico_dynam}).

The paper is structured as follows. In Section~\ref{sec_polimomios_n} we recall the main properties of the operators obtained applying the Chebyshev-Halley methods to the polynomials $z^n+c$ and analyse the stability of the strange fixed points. In Section~\ref{sec_Conectividad} we provide a dynamical condition for the simple connectivity of the basins of attraction of the roots and the connectivity of the Julia set. In Section~\ref{sec_bifurcation} we study different bifurcation parameters within the Collar of the Cat set. Finally, in Section~\ref{sec_Numerico} we study the parameter and the dynamical planes of the family from a numerical point of view.

\section{Dynamical study on degree \textit{n} polynomials \label{sec_polimomios_n}}

We study is the Chebyshev-Halley family of
numerical methods applied on the polynomials $z^{n}+c$. Results of this
family applied on polynomials of degree two can be seen in \cite{gato}, \cite{bulbosCheby}, \cite{periododobleCheby}.

In \cite{jordi} we show that the operator obtained for $f\left( z\right)
=z^{n}-1$ is conjugate to the one obtained when applying the
Chebyshev-Halley method to $f\left( z\right) =z^{n}+c$, $c\in \mathbb{C}
-\{0\}$. Therefore, it is enough to study the Chebyshev-Halley methods for $
z^{n}-1$ to understand them for all $z^{n}+c$, where $c\in \mathbb{C}-\{0\}$.

We denote by $O_{n,\alpha }(z)$ the fixed point operator  obtained  for $
z^{n}-1.$ By substituting $f\left( z\right) =z^{n}-1$ in (\ref{metodo})\ we have:

\begin{equation*}
O_{n,\alpha }(z)=z-\frac{(z^{n}-1)((-1+2\alpha +n-2\alpha n)+(1-2\alpha
-3n+2\alpha n)z^{n})}{2nz^{n-1}(\alpha (n-1)(z^{n}-1)-nz^{n})}=
\end{equation*}
\begin{equation}
=\frac{(1-2\alpha )(n-1)+(2-4\alpha -4n+6\alpha n-2\alpha
n^{2})z^{n}+(n-1)(1-2\alpha -2n+2\alpha n)z^{2n}}{2nz^{n-1}(\alpha
(1-n)+(-\alpha -n+\alpha n)z^{n})}.  \label{Opn}
\end{equation}

\noindent Except for degenerate cases studied in Section~\ref{sec_bifurcation}, the
degree of this operator is $2n$.  Therefore, it has $2n+1$  fixed points and $4n-2$ critical points.

The fixed points are
obtained by solving $O_{n,\alpha }(z)=z$. On one hand, we obtain the $n$th
-roots of the unity, corresponding to the zeros of the polynomial $z^{n}-1$,
which are superattracting fixed points. The other $n+1$ fixed points are $
z=\infty $ and the $n$ solutions of the equation

\begin{equation}
-1+2\alpha +n-2\alpha n+(1-2\alpha -3n+2\alpha n)z^{n}=0.  \label{fijosExt}
\end{equation}
These points are called \emph{strange fixed points}, because they do not match with the solutions of
the polynomial.
The critical points of the operator $O_{n,\alpha }$ are the solutions of $
O_{n,\alpha }^{\prime }(z)=0$, where
\begin{equation}
O_{n,\alpha }^{\prime }(z)=\frac{(z^{n}\!-\!1)^{2}(n-1)(\alpha (1-2\alpha
)(n-1)^{2}+\left( 1-2n-2\alpha +2\alpha n\right) \left( -\alpha -n+\alpha
n\right) z^{n})}{2nz^{n}(\alpha \left( n-1\right) +\left( \alpha +n-\alpha
n\right) z^{n})^{2}}.  \label{eq:O'}
\end{equation}

\noindent The $n$th-roots of the unity are double critical points and, hence, are
superattracting fixed points of local degree 3. The point $z=0$ is a
critical point of multiplicity $n-2$ since it is mapped with degree $n-1$ to
$z=\infty$. This assertion follows from the term $1/z^{n-1}$ on Equation
(\ref{Opn}). The remaining  $n$  critical points are given by

\begin{equation}
c_{n,\alpha ,\xi }=c_{\xi }=\xi \left( \frac{\alpha (n-1)^{2}(2\alpha -1)}{
n(2n-1)-\alpha (4n-1)(n-1)+2\alpha ^{2}(n-1)^{2}}\right) ^{1/n},
\label{criticalpointsn}
\end{equation}

\noindent where $\xi $ denotes an $n$th-root of the unity, \i .e. $\xi^{n}=1$.
The existence of any stable behaviour of the dynamical system other
than the basins of attraction of the zeros of $z^{n}-1$ is controlled by
the orbits of these $n$ free critical points.

The order of convergence to the roots for all members the Chebyshev-Halley family is at least 3. The next result, which corresponds to \cite[Proposition 6.3]{jordi}, states that there is a single parameter for which the order of convergence increases to 4. In this case, the $n$th-roots of the unity are critical points of multiplicity three and are
superattracting fixed points of local degree 4.

\begin{proposition}\label{grado4}
 For $n\geq2$, the operator $O_{n,\alpha}$ has order of convergence 4 if and only if $$\alpha=\frac{2n-1}{3n-3}.$$
\end{proposition}

The following lemma, which corresponds to \cite[Lemma 6.2]{jordi}, states that the dynamics of the maps $O_{n,\alpha}$ is symmetric with respect to the $n$th-root of the unity. It follows from the lemma that the $n$ free critical orbits are symmetric and, therefore, it is enough to control one of them.

\begin{lemma}\label{conjucacionn}
Let $n\in\nat$ and let $\xi$ be an $n$th-root of the unity, i.e.\ $\xi^n=1$. Then $I_{\xi}(z)=\xi z$ conjugates $O_{n,\alpha}(z)$ with itself, i.e.\
$$ I_{\xi}\circ O_{n,\alpha}(z)=O_{n,\alpha}\circ I_{\xi}(z).$$
\end{lemma}

The stability of the fixed point $z=\infty$ was studied in \cite{jordi},
where we proved the following proposition.

\begin{proposition}
\label{caracterinfty} The fixed point $z=\infty$ satisfies the following
statements.

\begin{enumerate}
\item If $\left\vert \alpha -\frac{1-4n+5n^2}{2(n-1)(2n-1)}\right\vert <
\frac{n}{2(2n-1)}$, then $z=\infty $ is an attractor. In particular, it is a
superattracting fixed point if its multiplier is equal to $0$, i.e.\ $\alpha
=\frac{n}{n-1}.$

\item If $\left\vert \alpha -\frac{1-4n+5n^2}{2(n-1)(2n-1)}\right\vert =%
\frac{n}{2(2n-1)}$, then $z=\infty $ is an indifferent fixed point.

\item If $\left\vert \alpha -\frac{1-4n+5n^{2}}{2(n-1)(2n-1)}\right\vert >
\frac{n}{2(2n-1)} $, then $z=\infty $ is a repelling fixed point.
\end{enumerate}
\end{proposition}

We finish this section studying the stability of the other strange fixed points, which are given by the solutions of (\ref{fijosExt}).

\begin{proposition}
\label{caracterExt}The $n$ strange fixed points are given by
\begin{equation}\label{extraños}
    z^{\ast }_{\xi}=\xi\left(\frac{1-2\alpha -n+2\alpha n}{1-2\alpha -3n+2\alpha n}\right)^{1/n},
\end{equation}
where $\xi$ denotes an $n$th-root of the unity, and satisfy the following statements.

\begin{enumerate}
\item If $\left\vert \alpha -\frac{2n-1}{n-1}\right\vert <\frac{1}{2}$, then
they are attracting fixed points. In particular, \ they are superattracting fixed points
if \ $\alpha =\frac{2n-1}{n-1}.$

\item If $\left\vert \alpha -\frac{2n-1}{n-1}\right\vert =\frac{1}{2}$, then
they are indifferent fixed  points.

\item If $\left\vert \alpha -\frac{2n-1}{n-1}\right\vert >\frac{1}{2}$, then
they are repelling fixed points.
\end{enumerate}
\end{proposition}

\proof
The derivative of the operator is given in Equation (\ref{eq:O'}). Substituting \ $z^{n}=\frac{1-2\alpha -n+2\alpha n}{(1-2\alpha -3n+2\alpha n)
}$ in (\ref{eq:O'}), we have that
\begin{equation*}
O_{n,\alpha }^{\prime }(z^{\ast }_{\xi})=\frac{-2-2\alpha \left( n-1\right) +4n}{
n-1}.
\end{equation*}
Writing $\alpha =a+ib$ and developing the equation $\left\vert O_{n,\alpha
}^{\prime }(z^{\ast }_{\xi})\right\vert =1$ we obtain
\begin{equation*}
\left( a-\frac{2n-1}{n-1}\right) ^{2}+b^{2}=\frac{1}{4},
\end{equation*}
which is the equation of the circle centred at $\frac{2n-1}{n-1}$ with radius 1/2. The strange fixed points are attracting inside this circle, since
 $\left\vert O_{n,\alpha }^{\prime }(z^{\ast }_{\xi})\right\vert <1.$ They are superattracting if
 $ O_{n,\alpha }^{\prime }(z^{\ast }_{\xi}) =0$, which is satisfied  for
 $\alpha =\frac{2n-1}{n-1}$. Outside this circle they satisfy $\left\vert O_{n,\alpha }^{\prime }(z^{\ast})\right\vert >1$,
 so they are repelling. Finally, on the circle they satisfy  $\left\vert O_{n,\alpha }^{\prime }(z^{\ast }_{\xi})\right\vert =1$, so they are
indifferent fixed points.
\endproof

Proposition  \ref{caracterinfty} explains the evolution
of the location and size of the head of the Cat set. In particular, it shows that  the radius of the corresponding disk decreases to $\frac{1}{4}$ as  $n\rightarrow
\infty$.
On the other hand, in Proposition  \ref{caracterExt} we show that the body of
the Cat set is a disk of radius $\frac{1}{2}$, which does not depend on $n$.

\begin{figure}[h!]
  \centering
  \subfigure[\small{$n=3$} ]{
   \begin{tikzpicture}
    \begin{axis}[width=8cm,  axis equal image, scale only axis,  enlargelimits=false, axis on top]
      \addplot graphics[xmin=-3,xmax=3,ymin=-3,ymax=3] {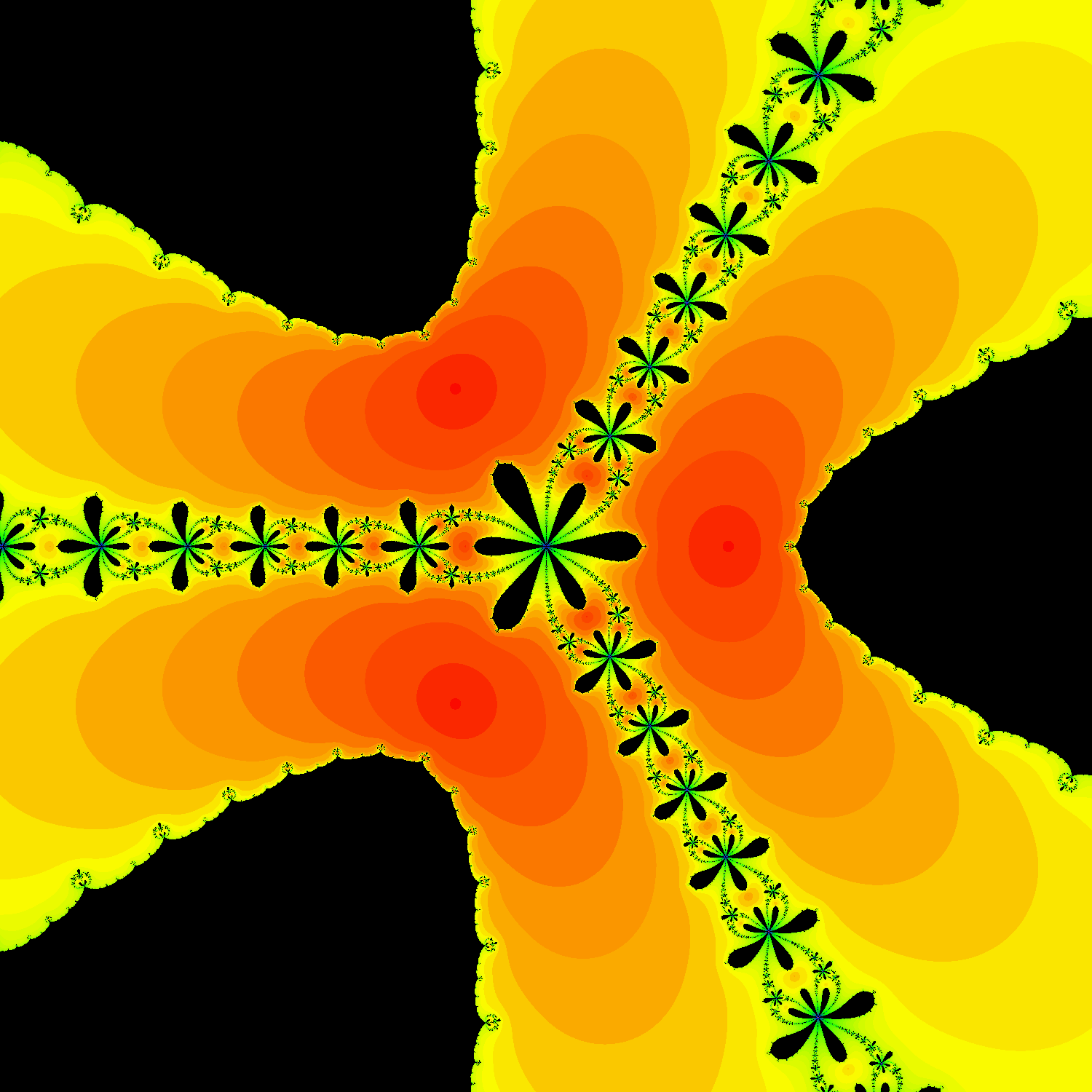};
    \end{axis}
  \end{tikzpicture}
    }
  \subfigure[\small{$n=10$} ]{
 \begin{tikzpicture}
    \begin{axis}[width=8cm,  axis equal image, scale only axis,  enlargelimits=false, axis on top]
      \addplot graphics[xmin=0,xmax=2,ymin=-1,ymax=1] {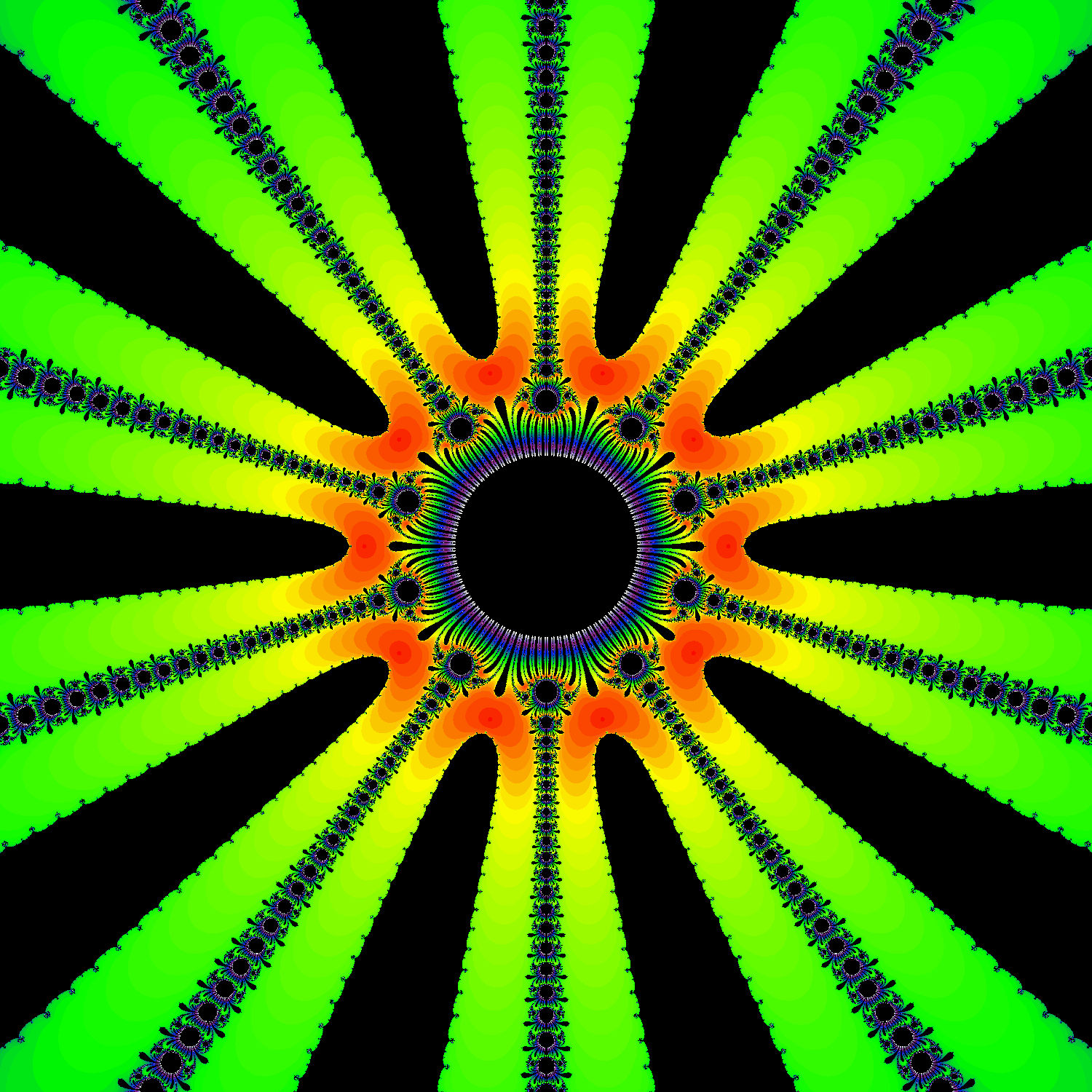};
    \end{axis}
  \end{tikzpicture}
    }
  \caption{Dynamical planes of $O_{n,\alpha}$ for $\alpha=\frac{2n-1}{n-1}$.\label{fig:dynamcuerpo}}
\end{figure}

In Figure~\ref{fig:dynamcuerpo} we show the dynamical planes of  $O_{n,\alpha}$ for $n\in\{3,10\}$ and $\alpha=\frac{2n-1}{n-1}$, which correspond to the superattracting case in Proposition \ref{caracterExt}. We observe with a scaling from red to green to purple and to grey the points which converge to the the roots in 75 iterates, that is, the iterate $w$ satisfies $|w-\xi|<10^{-4}$ for some $n$th-root of the unity $\xi$. We observe in black the points which have  not converged to the roots after 75 iterates. The basins of attraction of the strange fixed points appear as black fingers coming from infinity. Notice that the strange fixed points, and the corresponding basins of attraction,  are symmetrically located. We refer to Section~\ref{sec_Numerico_dynam} for a more detailed explanation on how  the images are produced.

\section{Simple connectivity of Fatou components} \label{sec_Conectividad}

Numerical experiments show that the speed of convergence to the roots may be
slower when the basins of attraction are not simply connected, that is,
there are connected components of the basins of attraction which have holes.
In this section we give a dynamical condition which characterizes whether the
immediate basins of attraction for the operators $O_{n,\alpha}$ are simply
connected. We also prove that the operators $O_{n,\alpha}$ can not have
Herman rings, that is, doubly connected rotation domains. The existence of
such domains would lead to a positive measure set of initial conditions not
converging to the roots. For completeness, we end this section showing how
the previous results lead to a dynamical characterization of the
connectivity of the corresponding Julia sets.

\subsection{Simple connectivity of immediate basins of attraction of superattracting fixed points}

In this subsection we give  a characterization of the simple connectivity of the immediate basin of
attraction of a supperattracting fixed point in the case that it can contain
at most one critical point besides the supperattracting fixed point itself. This result is stated for
a general rational map $f$. For simplicity, we assume that it has a
supperattracting fixed point at $z=0$. We denote by $A_f(0)$ its basin of
attraction and by $A^*_f(0)$ its immediate basin of attraction, that is, the connected component of $A_f(0)$ which contains $z=0$.

\begin{proposition}
\label{propconn} Let $f:\widehat{\mathbb{C}}\rightarrow\widehat{\mathbb{C}}$
be a rational map and let $z=0$ be a superattracting fixed point of $f$.
Assume that $A_f(0)$ contains at most one critical point other than $z=0$.
Then, exactly one of the following statements  holds.

\begin{enumerate}
\item The set $A_f^*(0)$ contains no critical point other than $z=0$. Then,
 $A_f^*(0)$ is simply connected.

\item The set $A_f^*(0)$ contains a critical point $c\neq 0$ and a preimage
$z_0$ of $z=0$, $z_0\neq 0$. Then, $A_f^*(0)$ is simply connected.

\item The set $A_f^*(0)$ contains a critical point $c\neq 0$ and no preimage
of $z=0$ other than $z=0$ itself. Then, $A_f^*(0)$ is multiply connected.
\end{enumerate}
\end{proposition}

The proof of Proposition~\ref{propconn} uses the Riemann-Hurwitz formula
(see \cite{beardon}). Given a proper map $f:U\rightarrow V$, this formula
relates the connectivities of $U$ and $V$. The connectivity of an open set
$U\subset \widehat{\mathbb{C}}$ is given by the number of connected
components of $\partial U$.
A map $f$ is proper if the
preimages of compact sets are compact sets. Given a rational map $f$, an open
set $V$ and a connected component $U$ of $f^{-1}(V)$, the map $f:U\rightarrow V$ is proper.

\begin{theorem}[Riemann-Hurwitz Formula]
\label{riemannhurwitz} Let $U$ and $V$ be two connected domains of
$\widehat{\mathbb{C}}$ of finite connectivity $m_{U}$ and $m_{V}$ and let
$f:U\rightarrow V$ be a degree $k$ proper map branched over $c$ critical
points counted with multiplicity. Then
\begin{equation*}
m_{U}-2=k(m_{V}-2)+c.
\end{equation*}
\end{theorem}

\begin{proof}[Proof of Proposition~\ref{propconn}]

We consider the B\"ottcher coordinate of the superattracting fixed point $z=0$ (see \cite[Theorem 9.1]{Milnor}).
The B\"ottcher coordinate is a conformal map $\phi:U\rightarrow\dis_r$,
 where $U$ is a simply connected neighbourhood of $z=0$, $0<r\leq1$, and $\dis_r$ is the open disk
 of center 0 and radius $r$. This map conjugates $f$ with $z\rightarrow z^d$, where $d$
 is the local degree of the supearattracting fixed point $z=0$. More specifically,
  for all $z$ with $|z|<r$ we have $\phi \circ f \circ \phi^{-}(z)=z^n$. In particular, $\phi(0)=0$.

We now assume that $U$ is the maximal domain of definition of the B\"ottcher coordinate and, hence, $r$ is maximal.
If $A^*_f(0)$ does not contain any extra critical point, then $r=1$ and $U=A^*_f(0)$ (see \cite[Theorem 9.3]{Milnor}).
 In particular, $A^*_f(0)$ is simply connected. On the other hand, if $A^*_f(0)$ contains an extra critical point $c\neq 0$,
 then $r<1$, $U \subseteq A^*_f(0)$, and $c\in\partial U$ (see \cite[Theorem 9.3]{Milnor}).
 Let $V=f(U)$. Then, $ \gamma_0:= \partial V$ is a simple closed curve
 (it is the preimage under $\phi$ of the circle or radius $r^d$).
  Denote by $\gamma_1$ the connected component of $f^{-1}(\gamma_0)$ which contains $\partial U$.
  Then, $\gamma_1$ is the union of two simple closed curves which intersect at the critical point $c$,
  say  $\gamma_1^1$, and $\gamma_1^2$. Furthermore, either $\partial U= \gamma_1^1$ or
  $\partial U= \gamma_1^1 \cup \gamma_1^2=\gamma_1$ (see Figure~\ref{fig:bottcher1}).

\begin{figure}[hbt!]
\centerline{
\setlength{\unitlength}{13cm}
 \begin{picture}(1,0.34267876)
    \put(0,0){\includegraphics[width=\unitlength,page=1]{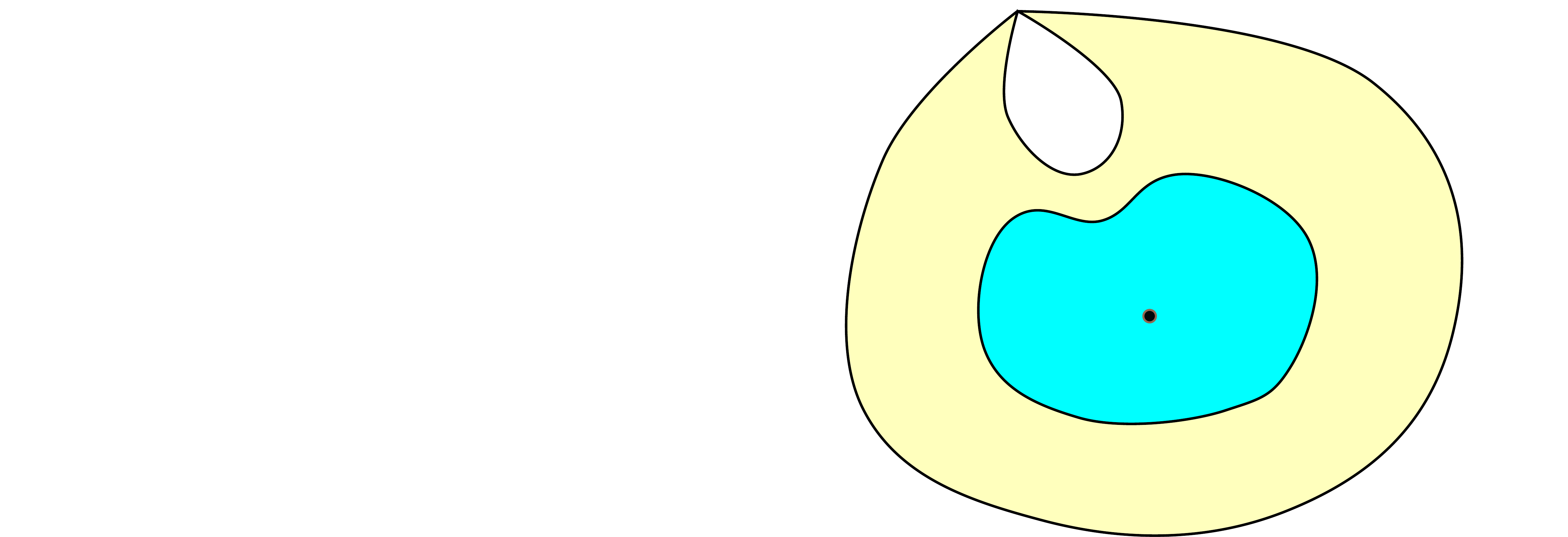}}
    \put(0.71094174,0.14){\color[rgb]{0,0,0}\makebox(0,0)[lb]{\smash{$0$}}}
    \put(0.84296273,0.15288058){\color[rgb]{0,0,0}\makebox(0,0)[lb]{\smash{$\gamma_0$}}}
    \put(0.72015251,0.28183133){\color[rgb]{0,0,0}\makebox(0,0)[lb]{\smash{$\gamma_1^2$}}}
    \put(0.94,0.2){\color[rgb]{0,0,0}\makebox(0,0)[lb]{\smash{$\gamma_1^1$}}}
    \put(0,0){\includegraphics[width=\unitlength,page=2]{maxdombottcher.pdf}}
    \put(0.04623146,0.268){\color[rgb]{0,0,0}\makebox(0,0)[lb]{\smash{$z_0$}}}
    \put(0.0846096,0.315){\color[rgb]{0,0,0}\makebox(0,0)[lb]{\smash{$\gamma_1^2$}}}
    \put(0.37,0.27){\color[rgb]{0,0,0}\makebox(0,0)[lb]{\smash{$\gamma_1^1$}}}
    \put(0.26,0.14827521){\color[rgb]{0,0,0}\makebox(0,0)[lb]{\smash{$0$}}}
    \put(0.33,0.18051289){\color[rgb]{0,0,0}\makebox(0,0)[lb]{\smash{$\gamma_0$}}}
    \put(0.20127933,0.25573415){\color[rgb]{0,0,0}\makebox(0,0)[lb]{\smash{$U$}}}
    \put(0.19820907,0.16669674){\color[rgb]{0,0,0}\makebox(0,0)[lb]{\smash{$V$}}}
    \put(0.77,0.17){\color[rgb]{0,0,0}\makebox(0,0)[lb]{\smash{$V$}}}
    \put(0.57,0.17130213){\color[rgb]{0,0,0}\makebox(0,0)[lb]{\smash{$U$}}}
  \end{picture}
}
\caption{The two possibilities for the configuration of $\gamma_1=\gamma_1^1\cup\gamma_1^2$. \label{fig:bottcher1}}
\end{figure}

Assume that $\partial U= \gamma_1^1$ and denote by $U'$ the connected component of $\wcom\setminus \gamma_1^2$ which does not contain $z=0$. Since $f(\gamma_1^1)= f(\gamma_1^2)=\gamma_0$, it follows that $f(U')=f(U)=V$. In particular $U'$ contains a preimage $z_0$ of $0$. Moreover, $\overline{U\cup U'}\subset A^*_f(0)$ and we can take a simply connected neighbourhood $\mathcal{U}\subset A^*_f(0)$ of $\overline{U\cup U'}$. Denote by $\mathcal{W}$ the connected component of $f^{-1}(\mathcal{U})$ that contains $0$.
Since $f$ does not have any critical point in $A^*_f(0)$ other than $c$ and $0$, it follows from the Riemann-Hurwitz formula (Theorem~\ref{riemannhurwitz}) that $f|_{\mathcal{W}}:\mathcal{W}\rightarrow \mathcal{U}$ has degree $d+1$ and $\mathcal{W}$ is simply connected.
 Notice that the degree of $f|_{\mathcal{W}}$ is at least $d+1$ since $z=0$ is a preimage of multiplicity $d$ of itself and $\mathcal{W}$ contains an extra preimage $z_0$ of $0$. Repeating the process, we obtain a sequence of simply connected domains $U=U_0\subset U_1\subset\cdots \subset U_n \subset \cdots$, where  $U_n$ is a connected component of $f^{-1}(U_{n-1})$  and $\overline{U_{n-1}}\subset U_n$. Moreover, $A^*_f(0)=\bigcup_{n=0}^{\infty} U_n$, from which we conclude that $A^*_f(0)$ is a simply connected set containing a preimage $z_0$ of $z=0$ and a critical point $c\neq 0$.

  Now assume that $\partial U= \gamma_1^1 \cup \gamma_1^2=\gamma_1$. Then, the connected components $W_1$ and $W_2$ of $\wcom\setminus \gamma_1^1$ and $\wcom\setminus \gamma_1^2$, respectively, which do not contain $z=0$ are mapped under $f$ onto open sets which contain $\wcom\setminus V$. In particular, both $W_1$ and $W_2$ contain part of the Julia set $\mathcal{J}(f)$ and, therefore, $A^*_f(0)$ is not simply connected. To finish the proof we have to see that, in this last case, $A^*_f(0)$ does not contain an extra preimage $z_0$ of $z=0$. Assume that this $z_0$ exists. Then,  $z_0\in W_1\cup W_2$.
  Denote by $\mathcal{A}_0$ the doubly connected set bounded by $\gamma_0$ and $\gamma_1$. Define recursively $\mathcal{A}_n$ as the union of all connected components of $f^{-1}(\mathcal{A}_{n-1})$ whose boundary has non-empty intersection with $\partial \mathcal{A}_{n-1}$. In particular, $\mathcal{A}_{1}$ consists of 2 connected components (see Figure~\ref{fig:bottcher2}). Since $A^*_f(0)$ has no critical point other than $c$ and $z=0$, it follows from the Riemann-Hurwitz formula (Theorem~\ref{riemannhurwitz}) that all connected components of $\mathcal{A}_{n}$ are doubly connected.

\begin{figure}[hbt!]
\centerline{
\setlength{\unitlength}{7cm}
\begin{picture}(1,0.79481965)
    \put(0,0){\includegraphics[width=\unitlength,page=1]{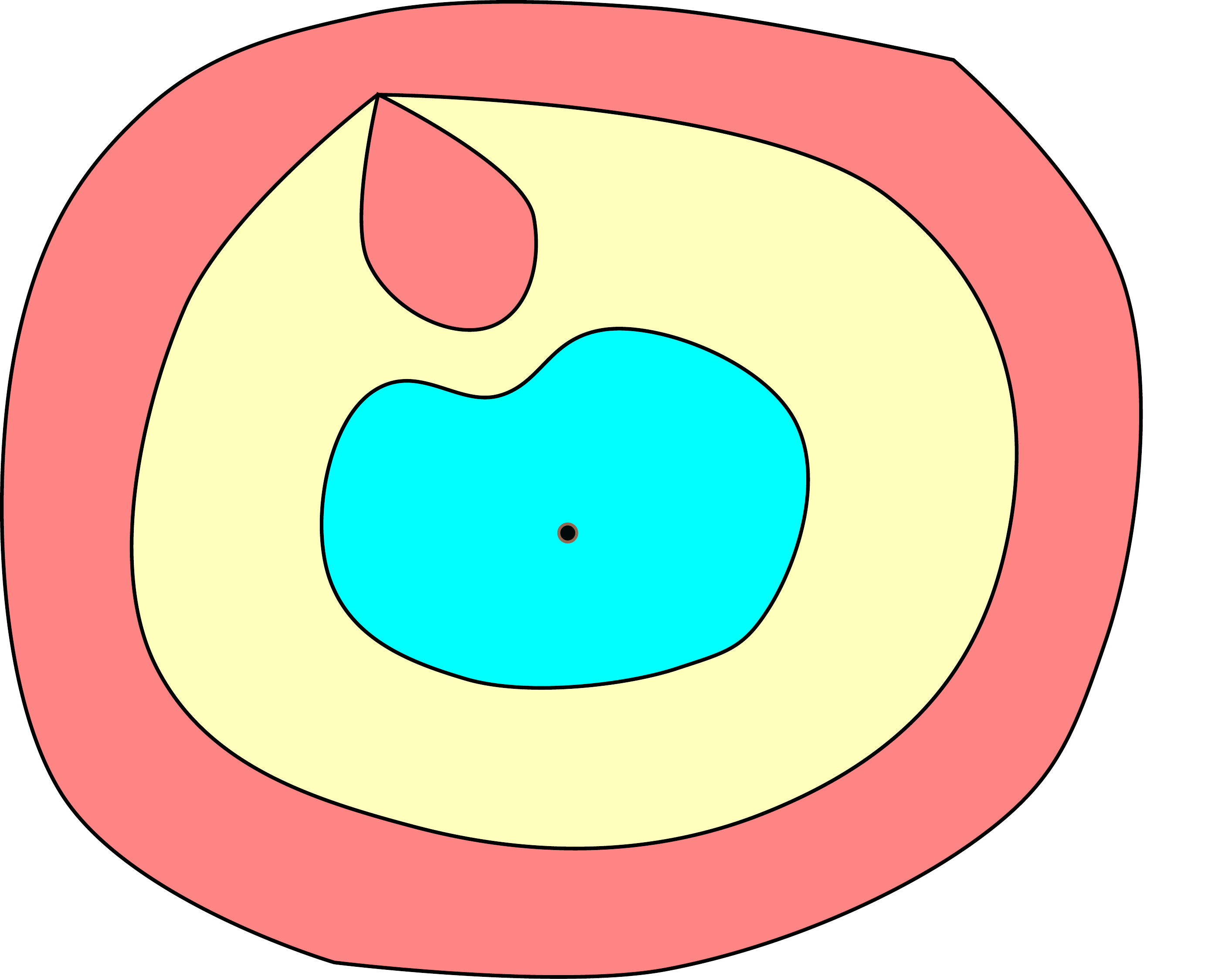}}
    \put(0.59,0.37551246){\color[rgb]{0,0,0}\makebox(0,0)[lb]{\smash{$\gamma_0$}}}
    \put(0.23,0.57){\color[rgb]{0,0,0}\makebox(0,0)[lb]{\smash{$\gamma_1^2$}}}
    \put(0.11694614,0.3811281){\color[rgb]{0,0,0}\makebox(0,0)[lb]{\smash{$\gamma_1^1$}}}
    \put(0.7,0.48501848){\color[rgb]{0,0,0}\makebox(0,0)[lb]{\smash{$\mathcal{A}_0$}}}
    \put(0,0){\includegraphics[width=\unitlength,page=2]{preimanillos.pdf}}
    \put(0.174,0.68){\color[rgb]{0,0,0}\makebox(0,0)[lb]{\smash{$\mathcal{A}_1$}}}
    \put(0.42,0.37){\color[rgb]{0,0,0}\makebox(0,0)[lb]{\smash{$0$}}}
    \put(0,0){\includegraphics[width=\unitlength,page=3]{preimanillos.pdf}}
    \put(0.32472674,0.60575586){\color[rgb]{0,0,0}\makebox(0,0)[lb]{\smash{$z_0$}}}
    \put(0.41,0.57){\color[rgb]{0,0,0}\makebox(0,0)[lb]{\smash{$\mathcal{A}_1'$}}}
  \end{picture}
}
\caption{Configuration of $\mathcal{A}_0$ and its preimage set $\mathcal{A}_1$. \label{fig:bottcher2}}
\end{figure}

  For $n$ small enough we can define $\mathcal{A}'_n$ as the connected component of $\mathcal{A}_n$ such that both connected components of $\partial \mathcal{A}'_n$ separate $z_0$ from $z=0$ (see Figure~\ref{fig:bottcher2}). If this component is well defined, it is unique. By construction, these components satisfy $\partial \mathcal{A}'_n\cap\partial \mathcal{A}'_{n-1}\neq \emptyset$. Indeed, there is a connected component $\zeta_n$ of $\partial \mathcal{A}'_n$ which is a simple closed curve and satisfies $\zeta_n\subset \partial \mathcal{A}'_{n-1}$ (see Figure~\ref{fig:bottcher2}).
  Since the sets $\mathcal{A}_n$ accumulate on $\partial A^*_f(0)$, the sets $\mathcal{A}'_n$ are not well defined for $n$ large enough. Let $N$ be the smallest such $n$.
  Then, there exists a unique connected component $\mathcal{A}'_N$ of $\mathcal{A}_N$ such that one connected component of $\partial \mathcal{A}'$ separates $z_0$ from $z=0$ while the other connected component of $\partial \mathcal{A}'_N$ does not separate them. As before, there is a connected component $\zeta_N$ of $\partial \mathcal{A}'_N$ which is a simple closed curve and satisfies $\zeta_N\subset \partial\mathcal{A}'_{N-1}$. This component $\zeta_N$ is precisely the one that separates $z_0$ from $z=0$. It follows that $z_0$ belongs to the doubly connected domain $\mathcal{A}'_N$. This is a contradiction since $f^N(\mathcal{A}'_N)\subset \mathcal{A}_0$ and $f^N(z_0)=0$ ($f(z_0)=0$).
\end{proof}

\begin{remark}
The immediate basin of attraction of an attracting or parabolic cycle can
only have connectivity 1 or $\infty$ (see \cite{beardon}). Therefore, when
it is multiply connected it has connectivity $\infty$.
\end{remark}

\subsection{Application to the Chebyshev-Halley family: Connectivity of the Julia set}

In this subsection we apply Proposition~\ref{propconn} to the operators $O_{n,\alpha}$.
All the proofs of this section rely on the symmetry property of the
operators $O_{n,\alpha}$ introduced in Lemma~\ref{conjucacionn}.  This lemma
tells us that if $\xi$ is an $n$th-root of the unity and $I_{\xi}(z)=\xi z$,
then $I_{\xi}$ conjugates $O_{n,\alpha}$ with itself. By using this
symmetry, any property which holds for the basin of attraction  of the root $1$ of $z^n-1$ also holds for the basins of attractions
of all the other $n$th-roots of the unity.

Recall that the roots of the unity are
superattracting fixed points of the operators $O_{n,\alpha}$, which
corresponds to the roots of the polynomial $z^n-1$. For simplicity, all
results regarding the basins of attraction of the roots are stated for $z=1$.
We denote by $A_{n,\alpha}(1)$ the basin of attraction of
$z=1$ under $O_{n,\alpha}$ and by $A_{n,\alpha}^*(1)$ its immediate basin
of attraction, that is, the connected component of $A_{n,\alpha}(1)$ which
contains $z=1$.

\begin{lemma}
\label{onecrit} For all $\alpha\in\mathbb{C}$ and $n\geq2$, $A_{n,\alpha}(1)$
contains at most one critical point other than $z=1$.
\end{lemma}

\proof
The operators $O_{n,\alpha}$ have only $n$ free critical points given by $%
c_{n,\alpha,\xi}$ (see Equation (\ref{criticalpointsn})). Given an $n$th-root of the unity $\xi$, it follows from Lemma~\ref{conjucacionn} that the map $I_{\xi}(z)$
conjugates $O_{n,\alpha}$ with itself. In particular, if $A_{n,\alpha}^*(1)$
contains a critical point $c$, then the critical point $\xi\cdot c$ belongs
to $A_{n,\alpha}^*(\xi\cdot 1)$. Therefore, for all $j\in\{0,1,\cdots,n-1\}$, the immediate basin of attraction $A_{n,\alpha}^*(\xi^j)$ of $\xi^j\cdot 1$ contains the critical point $\xi^j\cdot c$. Since there are only $n$
different free critical points and for all $j,k\in \{0,1,\cdots,n-1\}$, $%
j\neq k$, we have $\xi^j\neq\xi^k$, we can conclude that $A_{n,\alpha}^*(1)$
contains at most one critical point. \endproof

The next result follows directly from Lemma~\ref{onecrit} and Proposition~\ref{propconn}. We use the fact that $A_{n,\alpha}(1)$ can contain at most one critical point (Lemma~\ref{onecrit}) to apply Proposition~\ref{propconn}, obtaining a characterization of the simple connectivity of the basins of attraction of the roots for the operator $O_{n,\alpha}$.

\begin{corollary}
\label{rootsimply} For fixed $n\geq2$ and $\alpha\in\mathbb{C}$, the
immediate basins of attraction of the roots of $z^n-1$ under $O_{n,\alpha}$
are multiply connected if and only if $A_{n,\alpha}^*(1)$ contains a
critical point $c\neq 1$ and no preimage of $z=1$ other than $z=1$ itself.
\end{corollary}

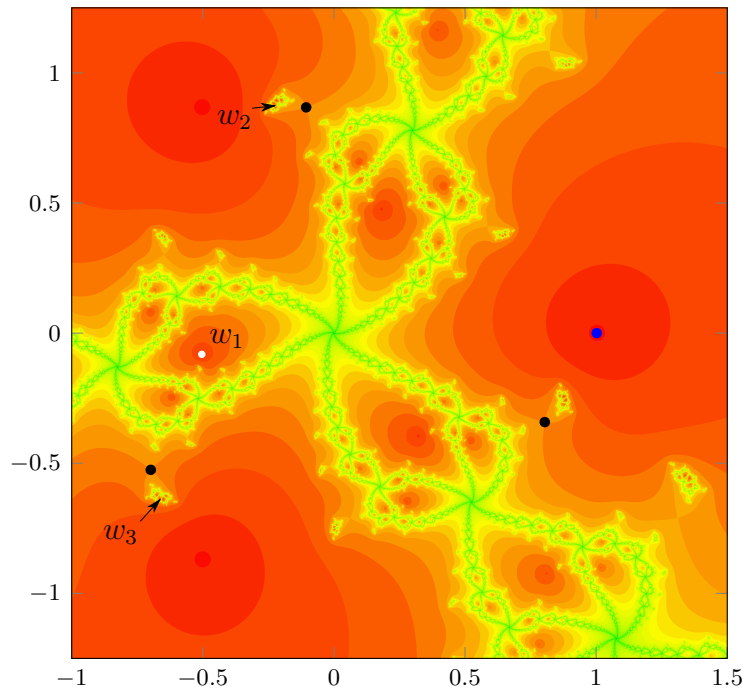
\begin{figure}[h!]
  \centering
  \setlength{\unitlength}{10cm}
     \begin{picture}(1,1)%
     \put(0,0){ \begin{tikzpicture}
    \begin{axis}[width=10cm,  axis equal image, scale only axis,  enlargelimits=false, axis on top]
      \addplot graphics[xmin=-1,xmax=1.5,ymin=-1.25,ymax=1.25] {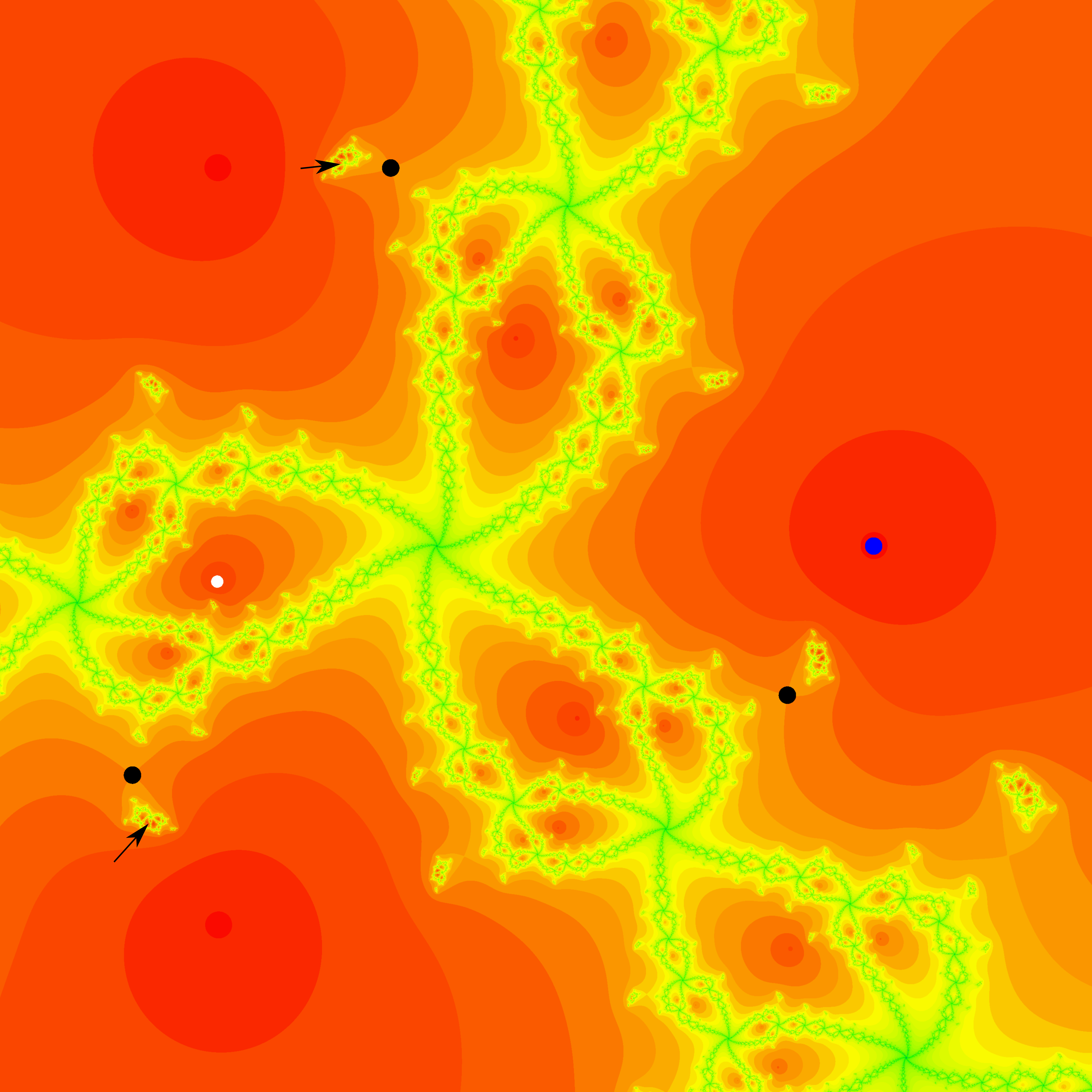};
       \end{axis}
       \end{tikzpicture}}
  % \put(0,0){\includegraphics[width=\unitlength]{disconCritZeros.pdf}}%
    \put(0.29,0.47){\color[rgb]{0,0,0}\makebox(0,0)[lb]{\smash{$w_1$}}}%
    \put(0.30,0.76){\color[rgb]{0,0,0}\makebox(0,0)[lb]{\smash{$w_2$}}}%
    \put(0.15,0.21){\color[rgb]{0,0,0}\makebox(0,0)[lb]{\smash{$w_3$}}}%
  \end{picture}
  \caption{Dynamical plane of $O_{3,\alpha}$ for $\alpha=0.2+1.592i$. \label{fig:dynambottcher}}
\end{figure}

In Figure~\ref{fig:dynambottcher} we show the dynamical plane of the operator $O_{3,\alpha}$ for $\alpha=0.2+1.592i$. The parameter is chosen so that the immediate basins of attraction of the roots are not simply connected. Since $n=3$, the map has degree $6$ and all points have 6 preimages. The roots are superattracting fixed points of local degree 3, so they are triple  preimages of themselves. Therefore, each root has 3 preimages other than itself.  In Figure~\ref{fig:dynambottcher} we show with a scaling from red to green the basins of attraction of the roots. We mark with black dots the locations of the three free critical points, and with a blue dot the location of the root $z=1$. We also mark the location of the 3 preimages of the root $z=1$, indicating them by $w_1$ (a white dot), $w_2$, and $w_3$. We observe that $A^*_{3,\alpha}(1)$ contains one critical point and no preimage of $z=1$. Hence, by Corollary \ref{rootsimply}, we know that the immediate basins of attraction of the roots are not simply connected. See Section~\ref{sec_Numerico_dynam} for a detailed description of how this picture has been produced.

The Julia set of a rational map is connected if and only if all Fatou components of the map are simply connected. We finish this section proving that only the Fatou components corresponding to the basins of attraction of the roots can have connectivity greater than one for the operators $O_{n,\alpha}$. We first prove that the operators $O_{n,\alpha}$ do not have Herman rings, that is, doubly connected rotation domains.

\begin{proposition}
\label{noherman} The maps $O_{n,\alpha}(z)$, where $n\geq 2$ and
$\alpha\in\mathbb{C}$, do not have Herman rings.
\end{proposition}

\proof
A cycle of Herman rings has, at least, two different infinite orbits of
critical points which accumulate on its boundary (see \cite{Shi1}). The idea
of the proof is the following: we  show that a map $O_{n,\alpha}$ is
semiconjugate to a rational map $S$ of the same degree which has a single
free critical orbit and, therefore, cannot have Herman rings.

Let $n\geq 2$ and $\alpha\in\mathbb{C}$ be fixed. Let $Q_n(z)=z^n$. It is
not difficult to see that there exists a rational map $S_{n,\alpha}$ such
that $Q_n\circ O_{n,\alpha}(z)=S_{n,\alpha}\circ Q_n(z)$. Therefore, the
maps $O_{n,\alpha}$ and $S_{n,\alpha}$ are semiconjugate.  Two points $z_1$ and $z_2$
and mapped to the same point under $Q_n(z)$ if and only if $z_1=\xi z_2$,
where $\xi^n=1$, and therefore the dynamics of $z_1$ and $z_2$ under $%
O_{n,\alpha}$ are conjugate (see Lemma~\ref{conjucacionn}). It follows that
$Q_n$ sends periodic points to periodic points (possibly of different period)
and preperiodic points to preperiodic points.
Moreover, a rational semiconjugacy between rational maps
sends (periodic) Fatou components onto (periodic) Fatou components. In
particular, if $H$ is a Herman ring of $O_{n,\alpha}$, then $Q_n(H)$ is a Herman ring of  $S_{n,\alpha}$.

To finish the proof we show that $S_{n,\alpha}$ can not have Herman rings.
Deriving the equation $Q_n\circ O_{n,\alpha}(z)=S_{n,\alpha}\circ Q_n(z)$ we
obtain

\begin{equation*}
n (O_{n,\alpha}(z))^{n-1}\cdot
O^{\prime}_{n,\alpha}(z)=S^{\prime}_{n,\alpha}(z^n)\cdot (n z^{n-1}).
\end{equation*}

Therefore, the critical points of $S_{n,\alpha}$ are given by:

\begin{itemize}
\item the points 0 and $\infty$;

\item the images under $Q_n(z)$ of the zeros and poles of $O_{n,\alpha}$ (notice that the only critical points of $Q_n(z)$ are $z=0$ and $z=\infty$);

\item points of the form $c^n$ where $c$ is a critical point of $O_{n,\alpha}(z)$.
\end{itemize}

Since $Q_n$ sends periodic points to periodic points and preperiodic points
to preperiodic points, the point $\infty$ is fixed under $S_{n,\alpha}$
while $0$ is a preperiodic point (it is mapped to $\infty$). Therefore, all
free critical points of $S_{n,\alpha}$ are of the form $c^n$ where $c$ is a
critical point of $O_{n,\alpha}(z)$. There is a
$c_{n,\alpha}\in\widehat{\mathbb{C}}$ such that all free critical points of $O_{n,\alpha}(z)$
have the form $\xi \cdot c_{n,\alpha}$, where $\xi^n=1$
(see Equation~(\ref{criticalpointsn})). In particular, $S_{n,\alpha}$ has a single free
critical point given by $c_{n,\alpha}^n$. Therefore, $S_{n,\alpha}$ cannot
have Herman rings since a cycle of Herman rings has, at least, two different
infinite orbits of critical points accumulating on its boundary (see \cite{Shi1}).
\endproof

The next proposition deals
with the simple connectivity of  periodic basins of attraction. Notice that a cycle of Fatou components
is considered to be simply connected if all its connected components are
simply connected. The proof is analogous to the one of
\cite[Proposition 3.6]{CFG1} using the symmetry in the dynamics described in Lemma~\ref{conjucacionn}.

\begin{proposition}
\label{periodicsimply} Given a map $O_{n,\alpha}$, where $n\geq 2$ and
$\alpha\in\mathbb{C}$, the immediate basin of attraction of any attracting,
superattracting or parabolic cycle other than the $n$th-roots of the unity is
simply connected.
\end{proposition}

Together with the fact that Siegel disks are always simply connected,
Proposition~\ref{periodicsimply} finishes the description of the
connectivity of periodic Fatou components of $O_{n,\alpha}$. In the next
proposition we characterize the simple connectivity of all preperiodic Fatou
components.

\begin{proposition}
\label{preperiodicsimply}

Let $U$ be a preperiodic Fatou component of $O_{n,\alpha}$, where $n\geq 2$
and $\alpha\in\mathbb{C}$. Then, $U$ is simply connected if and only if it
is eventually mapped under $O_{n,\alpha}$ onto a simply connected periodic
Fatou component of $O_{n,\alpha}$ .
\end{proposition}

\proof
The preimage of a multiply connected Fatou component under a rational map is
also multiply connected. Therefore, if $U$ is eventually mapped onto a
multiply connected periodic Fatou component, then $U$ is multiply connected.

We know from the Riemann-Hurwitz formula (Theorem~\ref{riemannhurwitz}) that
if a rational map $f$ maps a multiply connected Fatou component $U$ onto a
simply connected Fatou component $V$, then $U$ contains, at least, two
distinct critical points. It follows that the preimages of all simply
connected periodic Fatou components of $O_{n,\alpha}$ are simply connected.
Indeed, from Lemma~\ref{onecrit} we know that the basins of attraction of
the roots of $z^n-1$ contain, at most, one critical point and, therefore, if
the immediate basin of attraction of the root is simply connected, then so
are its preimages. On the other hand, any periodic Fatou component other
than the basins of attraction of the roots are related to the free critical
points. Attracting, superattracting and parabolic cycles contain, at least,
one critical point while Siegel disks have the orbit of, at least, a
critical point accumulating on its boundary (see, for instance \cite{beardon, Milnor}). By Lemma~\ref{conjucacionn}, if one free critical point
is related to a periodic Fatou component, then so are all of them. In
particular, no preimage of a periodic Fatou component other than the
immediate basin of attraction of a root can contain a critical point and,
hence, all preperiodic Fatou components which do not belong to the basin of
attraction of a root are simply connected.

\endproof

The next result follows directly from Corollary~\ref{rootsimply},
Proposition~\ref{noherman}, Proposition~\ref{periodicsimply}, and
Proposition~\ref{preperiodicsimply} since the Julia set of a rational map is
connected if and only if all its Fatou components are simply connected.

\begin{theorem}
\label{thmconJulia}

For fixed $n\geq 2$ and $\alpha \in \mathbb{C}$, the the Julia set
$\mathcal{J}(O_{n,\alpha })$ of $O_{n,\alpha }$ is disconnected if and only if
 $A_{n,\alpha }^{\ast }(1)$ contains a critical point $c\neq 1$ and no
preimage of $z=1$ other than $z=1$ itself.
\end{theorem}

It follows directly from the previous theorem that the parameters of the Cat set correspond to operators with a connected Julia set. Recall that the Cat set is defined as the set of parameters for which the free critical points of the operator do not belong to the basins of attraction of the roots. In particular, we obtain the next result.

\begin{corollary}\label{corollary_con}
 Let  $O_{n,\alpha}$ be an operator with a parabolic cycle, a Siegel disk, or an attracting cycle other than the $n$th-roots of the unity. Then, its Julia set is connected.
\end{corollary}

\proof
Every attracting or parabolic cycle has a critical point on its basin of attraction. Also, given a Siegel disk there exists a critical point whose orbit accumulates on its boundary (see \cite{beardon}). Therefore, if there is an stable domain different
from the basins of attraction of the roots, at least one free critical point is related to the domain and, hence, cannot lie inside the basin of attraction of any of the roots. It follows from the symmetry of the family (see Lemma~\ref{conjucacionn}) that no critical point belongs to the basin of attraction of a root and, by Theorem~\ref{thmconJulia}, the Julia set is connected.
 \endproof

The head and the body of the Cat set are examples of sets  of parameters for which we can apply Corollary~\ref{corollary_con}. If $\left\vert \alpha -\frac{1-4n+5n^2}{2(n-1)(2n-1)}\right\vert <\frac{n}{2(2n-1)}$,
the fixed point $z=\infty$ is an attracting
fixed point (Proposition \ref{caracterinfty}). So, at least one free
critical point must be in its basin of attraction. As the symmetry of the
problem is preserved, all the free critical points must be in this basin of
attraction and there are no free critical points in the basins of the roots. Similarly, if
 $\left\vert \alpha -\frac{2n-1}{n-1}\right\vert < \frac{1}{2}$ the $n$
strange fixed points are attracting  (Proposition
\ref{caracterExt}) and, hence, each of them has a critical point in its basin of
attraction. Therefore, each free critical point belongs to the basin of
attraction of one of the strange fixed points and there are no extra critical points in the basins
of the roots.

 We can analyse the connectivity of the Julia set for the hyperbolic component which contains the parameter $\alpha=\frac{2n-1}{3n-3}$. Recall that this parameter corresponds to an operator $O_{n,\alpha}$ with order of convergence 4 (see Proposition~\ref{grado4}).  A hyperbolic component in the parameter plane is an open set of parameters for which the orbits of all critical points converge to attracting or superattracting cycles.

\begin{corollary} \label{corolario_collar}
For fixed $n\geq2$, the Julia set of the operator $O_{n,\alpha}$ is connected
for all the parameters in the hyperbolic component which contains $\alpha=\frac{2n-1}{3n-3}$.
\end{corollary}

\proof
 The Julia set is connected for one parameter of a hyperbolic component if and only if it is connected for all the  parameters of the component (see \cite{MSS}). Hence, it is enough to prove that the Julia set is connected for $\alpha=\frac{2n-1}{3n-3}$.

 By Proposition~\ref{grado4},  the local degree of the roots for the operator $O_{n,\alpha}$ is 4  for $\alpha=\frac{2n-1}{3n-3}$. In particular, for $\alpha=\frac{2n-1}{3n-3}$ the free critical points coincide with the roots and the basins of attraction cannot contain any critical point other than the roots. Therefore, the Julia set is connected by Theorem~\ref{thmconJulia}.

 \endproof

%If $\alpha\neq \frac{2n-1}{3n-3}$ is close to $\frac{2n-1}{3n-3}$,  the roots have local degree 3. However, for all $\alpha$ in the same hyperbolic component than $\frac{2n-1}{3n-3}$, the immediate basins of attraction of the roots are still mapped with degree 4 onto themselves. Since the roots are triple preimages of themselves, it follows that the immediate basin of attraction of each root contains an extra preimage of the root.  We conclude  by Theorem~\ref{thmconJulia} that the Julia set is connected.

\section{Bifurcation parameters}\label{sec_bifurcation}

The goal of this section is to study parameters which are in the bifurcation locus of the parameter planes of the Chebyshev-Halley family operator $O_{n,\alpha }$. We focus on two types of such parameters.
  On one hand, we study the degeneracy parameters for which the operator decreases its degree. On the other hand, we search parameters for which the free critical points are eventually map to the fixed point $z=\infty$ and, hence, they are preperiodic.

We first study the parameters for which the operator  $O_{n,\alpha }$ decreases its degree.

\begin{lemma}\label{lemmahalley}
For $\alpha =\frac{1}{2}$ and $\alpha =\frac{2n-1}{2n-2},$ the family $
O_{n,\alpha }(z)$ degenerates to a rational function of lower degree.
\end{lemma}

\proof Consider the rational function $O_{n,\alpha }(z)$ given in (\ref{Opn}):
\begin{equation*}
O_{n,\alpha }(z)=\frac{(1-2\alpha )(n-1)+(2-4\alpha -4n+6\alpha n-2\alpha
n^{2})z^{n}+(n-1)(1-2\alpha -2n+2\alpha n)z^{2n}}{z^{n-1}\left( 2\alpha
n(1-n)+2n(-\alpha -n+\alpha n)z^{n}\right) }.
\end{equation*}

The degree of a rational function decreases if the higher degree term vanishes or if the roots of the numerator and denominator match.
Making the higher degree term $(1-2\alpha -2n+2\alpha n)$ equal to to zero, we obtain the value
\begin{equation*}
\alpha =\frac{2n-1}{2n-2}.
\end{equation*}

For $\alpha \neq \frac{2n-1}{2n-2}$, we calculate the roots of numerator and
denominator.
The roots of the numerator are given by

\begin{equation}
z=\left( \frac{\left( 2n-1+\alpha \left( n-1\right) \left( n-2\right)
\right) \pm n\sqrt{\left( \alpha \left( n-1\right) -1\right) ^{2}+2\left(
n-1\right) }}{(1-2n+2\alpha (n-1))(n-1)}\right) ^{1/n}.  \label{r_num}
\end{equation}

The roots of the denominator are $z=0$ \ and
\begin{equation}
z=\left( \frac{\left( 1-n\right) \alpha }{n+\alpha \left( 1-n\right) }\right) ^{1/n},
\text{ \ for }\alpha \neq \frac{n}{n-1}.  \label{r_den}
\end{equation}

The case $\alpha=n/(n-1)$ is studied separately. The values of $\alpha $ for which the roots given in (\ref{r_num})
coincide with $z=0$ are solutions of the equation
\begin{equation*}
\left( 2n-1+\alpha \left( n-1\right) \left( n-2\right) \right) \pm
n\sqrt{\left( \alpha \left( n-1\right) -1\right) ^{2}+2\left( n-1\right) }=0.
\end{equation*}
 Operating on the expression we obtain
\begin{equation*}
\left( n\sqrt{\left( \alpha \left( n-1\right) -1\right) ^{2}+2\left(
n-1\right) }\right) ^{2}=\left( 2n-1+\alpha \left( n-1\right) \left(
n-2\right) \right) ^{2},
\end{equation*}
and \bigskip $\left( 2\alpha -1\right) \left( n-1\right) ^{2}\left(
2n+2\alpha -2n\alpha -1\right) =0.$ Thus, the solutions are
\begin{equation*}
\alpha =\frac{1}{2}\text{ \ and \ }\alpha =\frac{2n-1}{2n-2}.
\end{equation*}

Since we are considering the case $\alpha \neq \frac{2n-1}{2n-2},$ we obtain $\alpha =\frac{1}{2}.$
Secondly, we look for values of $\alpha $ so that roots given in (\ref{r_num})
coincide with roots given in (\ref{r_den}). They are solutions of the equation
\begin{equation*}
\frac{\left( 2n-1+\alpha \left( n-1\right) \left( n-2\right) \right) \pm n
\sqrt{\left( \alpha \left( n-1\right) -1\right) ^{2}+2\left( n-1\right) }}
{(1-2n+2\alpha (n-1))(n-1)}=\frac{\left( 1-n\right) \alpha }{n+\alpha \left(
1-n\right) }.
\end{equation*}

Operating on the expression we obtain

\begin{equation*}
\left( \frac{\left( 1-n\right) \alpha (1-2n+2\alpha (n-1))(n-1)}{n+\alpha
\left( 1-n\right) }-\left( 2n-1+\alpha \left( n-1\right) \left( n-2\right)
\right) \right) ^{2}=
\end{equation*}

\begin{equation*}
=\left( n\sqrt{\left( \alpha \left( n-1\right) -1\right) ^{2}+2\left(
n-1\right) }\right) ^{2},
\end{equation*}
which implies
\begin{equation*}
\left( 1+2\alpha \left( n-1\right) -2n\right) \left( n-1\right) ^{2}n^{2}=0.
\end{equation*}
The root  of the last equation is
\begin{equation*}
\alpha =\frac{2n-1}{2n-2}.
\end{equation*}
However, this is not a valid solution since we are considering the case $\alpha \neq \frac{2n-1}{2n-2}$.

Finally,  we consider the value $\alpha =\frac{n}{n-1}$ discarded in Equation (\ref{r_den}). In this case, the degree of the operator remains $2n$:
\begin{equation*}
O_{n,\frac{n}{n-1}}(z)=\frac{\left( n+1\right) +2(n^{2}-1)z^{n}-\left(
n-1\right) z^{2n}}{2n^{2}z^{n-1}}.
\end{equation*}

Then, the only values for which the operator degenerates to a rational
function of lower degree are $\alpha =\frac{1}{2}$ and $\alpha =\frac{2n-1}{2n-2}.$
\endproof

In the following we analyse the dynamical behaviour of the numerical methods  corresponding to the bifurcation parameters obtained above.

\begin{lemma}  \label{0'1/2}
For $\alpha =\frac{1}{2}$ the strange fixed points are $z=0$ and $z=\infty$. Moreover, the only critical points are the roots of the polynomial, which have multiplicity $2$
\end{lemma}

\begin{proof}
The rational function for $\alpha =\frac{1}{2}$ is given by
\begin{equation}
O_{n,\frac{1}{2}}(z)=\frac{(n+1)z+(n-1)z^{n+1}}{(n-1)+(n+1)z^{n}}.
\label{ope1}
\end{equation}
The fixed points of $O_{n,\frac{1}{2}}(z)$ are the roots,  $z=0$ (the $n$ strange
fixed points collide at $0$) and $z=\infty$.
The derivative is given by
\begin{equation}
O_{n,\frac{1}{2}}^{\prime }(z) =\frac{\left( n^{2}-1\right) \left(
z^{n}-1\right) ^{2}}{((n-1)+(n+1)z^{n})^{2}}. \label{eq_0'1/2}
\end{equation}
The only critical points, which are the zeros of $O_{n,\frac{1}{2}}^{\prime }(z)$, are the roots of the polynomial $z^n-1$ and have multiplicity $2$.
\end{proof}

\begin{lemma}\label{lemmasingular}
For $\alpha =\frac{2n-1}{2n-2}$, the roots of the polynomial are critical points of multiplicity $2$. Moreover, the points $\{0,\infty \}$ form a cycle of period 2, which is superattracting if $n>2$.
\end{lemma}

\begin{proof}

The operator for $\alpha =\frac{2n-1}{2n-2}$ is
\begin{equation}
O_{n,\frac{2n-1}{2n-2}}(z)=\frac{1+(2n-1)z^{n}}{\left( 2n-1+z^{n}\right)
z^{n-1}}.  \label{ope2}
\end{equation}
The fixed points are the roots of the polynomial and  the $n$th-roots of $-1$.
The derivative is given by
\begin{equation}
O_{n,\frac{2n-1}{2n-2}}^{\prime }(z) =-\frac{\left( 2n^{2}-3n+1\right)
\left( z^{n}-1\right) ^{2}}{z^{n}\left( 2n-1+z^{n}\right) ^{2}}.
\end{equation}
The roots of the polynomial are  critical points with multiplicity $2$.
Since the numerator of $O_{n,\frac{2n-1}{2n-2}}(z)$ does not vanish at $z=0$ and  $z=0$ is a root of multiplicity $n-1$ of the denominator, we conclude that $z=0$ is mapped to $z=\infty$ with degree $n-1$.
Therefore, $z=0$ is a critical point if $n-1>1$, that is, $n>2$. Similarly, $z= \infty$ is mapped to $z=0$ with degree $n-1$ and is a critical point if $n>2$. In particular $\{0,\infty\}$ is a period two cycle. Since a periodic cycle is superattracting when any of the points of the cycle is critical, we conclude that $\{0,\infty\}$ is a superattracting cycle if $n>2$.
\end{proof}

Notice that the point $z=\infty$ is a fixed point for all $\alpha\neq\frac{2n-1}{2n-2}$ but it becomes a period 2 superattracting periodic point if $n>2$ and $\alpha=\frac{2n-1}{2n-2}$ (see Lemma~\ref{lemmasingular}). This may lead to very unstable dynamics (compare Figure~\ref{fig:paramzoomsing}).

We now look for bifurcation parameters for which the free critical points coincide with $z=0$ or are mapped onto it. Notice that $z=0$ is  a preimage of the  fixed point $z=\infty$. In that case, the free critical points are preperiodic and, therefore, there is no other stable behaviour than  the basins of attraction of the roots.

\begin{lemma}\label{solofijos}
 The free critical points collapse with $z = 0$ if and only if $\alpha =0$.
\end{lemma}

\proof
In the expression of the critical points (Equation (\ref{criticalpointsn}))  we observe that critical points become equal to 0 for $\alpha =0$ and $\alpha =\frac{1}{2}$.
Nevertheless, $\alpha =\frac{1}{2}$ is a degenerate case;  in this case, as we have seen in Lemma \ref{0'1/2}, $z=0$ is not a critical point. So, the only value for which free critical points collapse with $z = 0$ is  $\alpha =0$.
\endproof

Now, we calculate  the values of the parameter such that critical points are preimages of $z=0$. As before, for such parameters the critical points are preperiodic since
 $z=0$ maps to the  fixed point $z=\infty$.

\begin{lemma}\label{lemmaprecrit}
Critical points are preimages of $z=0$ if and only if  $\alpha =\frac{1\pm \sqrt{2\left(
1-n\right) }}{n-1}.$
\end{lemma}

\proof  For $\alpha \neq \{\frac{n}{n-1},\frac{2n-1}{2n-2}\}$, the free critical points are given by

\begin{equation}
c_{n,\alpha, \xi }=\xi\left( \frac{\left( n-1\right) ^{2}\left( -1+2\alpha \right)
\alpha }{n\left( 2n-1\right) -\left( 4n-1\right) \left( n-1\right) \alpha
+2\left( n-1\right) ^{2}\alpha ^{2}}\right) ^{1/n},
\end{equation}
where $\xi$ denotes an $n$th-root of the unity. We discard the parameter $\alpha=\frac{n}{n-1}$ since the critical points coincide with $z=\infty$ (see Proposition~\ref{caracterinfty}). We also discard the parameter $\frac{2n-1}{2n-2}$ since the degree of $O_{n,\alpha}$ decreases and there are no free critical points (see Lemma~\ref{lemmahalley} and Lemma~\ref{lemmasingular}). Preimages of $z=0$ are obtained from $O_{n,\alpha }\left( z\right) =0$,
that is
\begin{equation*}
(1-2\alpha )(n-1)+(2-4\alpha -4n+6\alpha n-2\alpha
n^{2})z^{n}+(n-1)(1-2\alpha -2n+2\alpha n)z^{2n}=0.
\end{equation*}
The solutions satisfy
\begin{equation*}
z^{n}=\frac{-1+2n+\alpha \left( n-1\right) \left( n-2\right) \pm
 n\sqrt{2\left( n-1\right) +\left( 1+\alpha \left( 1-n\right) \right) ^{2}}}{\left(
n-1\right) \left( 1-2n+2\alpha \left( n-1\right) \right) }.
\end{equation*}
The solutions coincide with critical points if

\begin{equation*}
\frac{-1+2n+\alpha \left( n-1\right) \left( n-2\right) \pm n\sqrt{2\left(
n-1\right) +\left( 1+\alpha \left( 1-n\right) \right) ^{2}}}{\left(
n-1\right) \left( 1-2n+2\alpha \left( n-1\right) \right) }=
\end{equation*}

\begin{equation*}
=\frac{\left( n-1\right) ^{2}\ \left( 2\alpha -1\right) \alpha }{n\left(
2n-1\right) -\left( 4n-1\right) \left( n-1\right) \alpha +2\left( n-1\right)
^{2}\alpha ^{2}}.
\end{equation*}

Operating and simplifying we obtain
\begin{equation*}
\left( 2\alpha -1\right) \left( 1+2\alpha \left( n-1\right) -2n\right)
^{2}\left( n-1\right) ^{2}\left( \alpha ^{2}\left( n-1\right) ^{2}-2\alpha
\left( n-1\right) +2n-1\right) =0.
\end{equation*}
The solutions of the last equation are
\begin{equation*}
\alpha =\frac{1}{2},\qquad \alpha =\frac{2n-1}{2\left( n-1\right) }, \qquad
\alpha =\frac{1\pm \sqrt{2\left( 1-n\right) }}{n-1}.
\end{equation*}
The parameters  $\alpha =\frac{1}{2}$ and $\alpha =\frac{2n-1}{ 2n-2 }$ are degenerated cases with no free critical points (see Lemma~\ref{lemmahalley}).

We conclude that critical points are the preimages of $z=0$ if and only if
$\alpha =\frac{1\pm \sqrt{2\left( 1-n\right) }}{n-1}$.
\endproof

Notice that the
values of $\alpha $ introduced in Lemma~\ref{lemmaprecrit} coincide, in the parameter space, with  two symmetric
points of the Collar of the Cat set, where the main  ramifications appear (see Figure~\ref{fig:paramzoombif}).

\section{Numerical studies and conclusions}\label{sec_Numerico}

The goal of this section is to perform a numerical study on the Chebyshev-Halley family. We first analyse the parameter planes near the bifurcation parameters described in Section~\ref{sec_bifurcation}. Afterwards we study the dynamical planes of the family. We focus on the dynamics of maps with a disconnected Julia set  and the evolution when $n$ grows of the  dynamics of some relevant members of the Chebyshev-Halley family.

\subsection{Parameter planes}\label{sec_Numerico_param}

 The drawings of parameter planes of the operators $O_{n,\alpha}$ along the paper are done using a program written in C which works as follows. We take a grid of points ($1500\times1000$ points for Figure~\ref{fig:paramgaton}, $1200\times2000$ points for Figure~\ref{fig:paramzoombif}, and $1500\times1500$ points for Figure~\ref{fig:paramzoomsing}). Then, we associate a parameter $\alpha\in\com$ to each point of the grid.
  The range of the real and the imaginary part of the parameters $\alpha$ is indicated in the horizontal and vertical axes of the images, respectively. For fixed $\alpha$ we compute one of the $n$ free critical points and iterate it up to 150 times. At each iteration we verify if the iterated point $w$ has converged to any of the $n$th-roots of the unity (we verify if $|w-\xi|<10^{-4}$, for any $n$th-root of the unity $\xi$). If the critical orbit converges to a root of the unity, we colour the corresponding point with an scaling from red (fast convergence) to yellow,  green, blue, purple and to grey (slow convergence). If after 150 iterations the orbit of the critical point has not converged to a root, we plot the point in black. In Figure~\ref{fig:paramgaton} we show the parameter planes of the operator $O_{n,\alpha}$ for different values of $n$.

\begin{figure}[hbt!]
\centering
\subfigure[\small{$n=10$} ]{
    \begin{tikzpicture}
    \begin{axis}[width=235pt, axis equal image, scale only axis,  enlargelimits=false, axis on top, %xtick={-3,-2,-1,0,1,2,3}, ytick={-3,-2,-1,0,1,2,3}
    ]
      \addplot graphics[xmin=-0.5,xmax=0.7,ymin=-1,ymax=1] {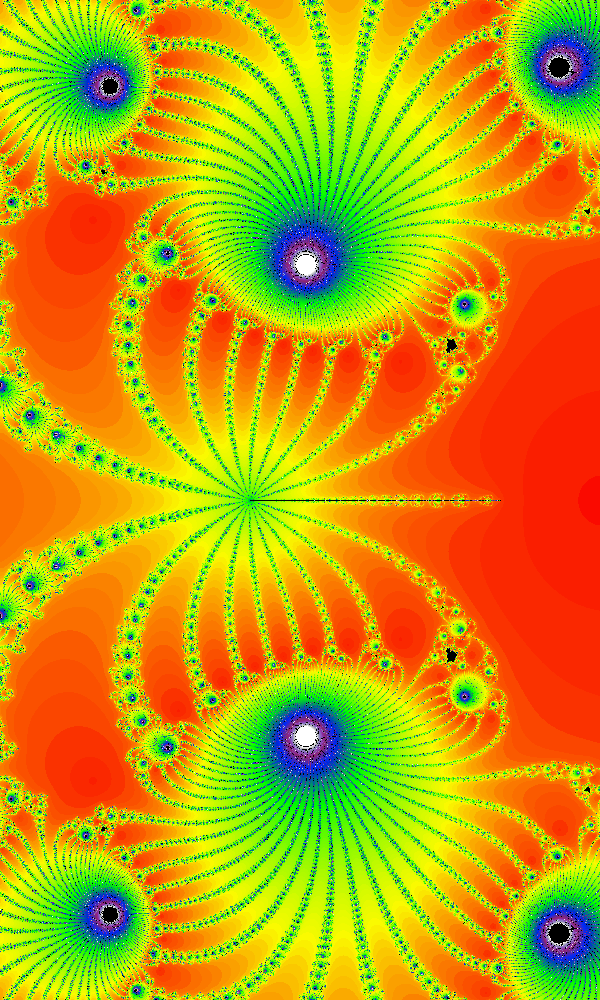};
    \end{axis}
  \end{tikzpicture}
    }
  \subfigure[\small{$n=25$} ]{
    \begin{tikzpicture}
    \begin{axis}[width=235pt, axis equal image, scale only axis,  enlargelimits=false, axis on top, %xtick={-3,-2,-1,0,1,2,3}, ytick={-3,-2,-1,0,1,2,3}
    ]
      \addplot graphics[xmin=-0.5,xmax=0.7,ymin=-1,ymax=1]{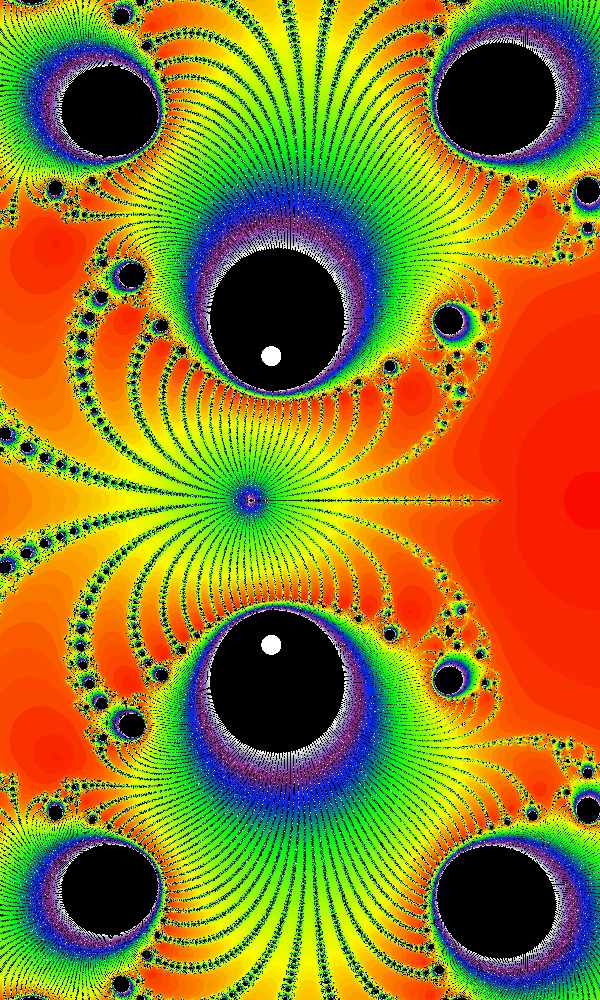};
    \end{axis}
  \end{tikzpicture}
    }
    \subfigure[\small{$n=100$}]{
    	\begin{tikzpicture}
    		\begin{axis}[width=235pt, axis equal image, scale only axis,  enlargelimits=false, axis on top, %xtick={-3,-2,-1,0,1,2,3}, ytick={-3,-2,-1,0,1,2,3}
    ]
      \addplot graphics[xmin=-0.5,xmax=0.7,ymin=-1,ymax=1] {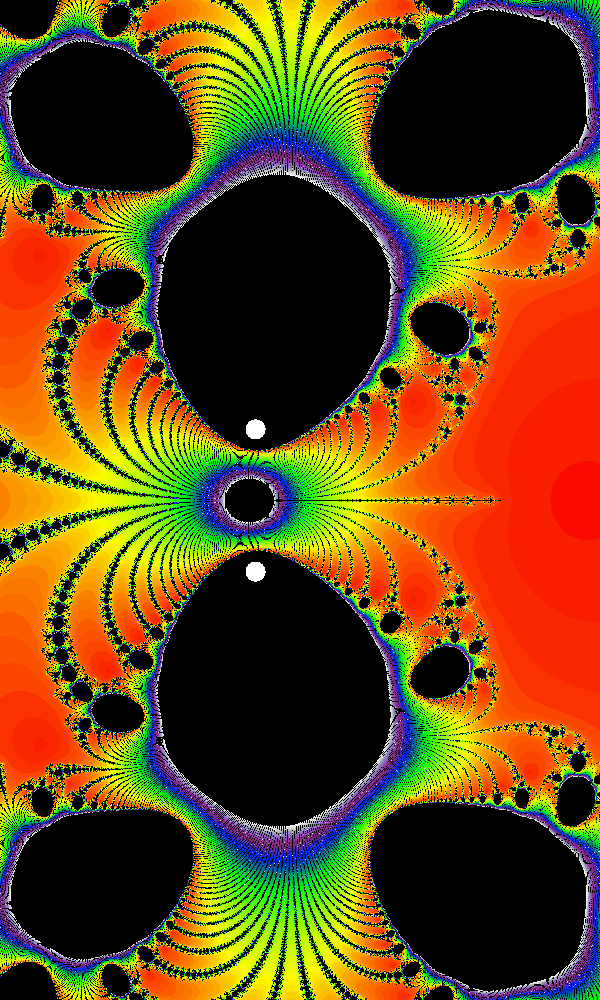};
    		\end{axis}
  		\end{tikzpicture}	
  }
      \caption{\small Zooms in on the parameter planes near the main bifurcation points of the Collar.\label{fig:paramzoombif}}
    \end{figure}

The aim of this subsection is to analyse the parameter planes of the family near the bifurcation parameters found in Section~\ref{sec_bifurcation}. In Figure~\ref{fig:paramzoombif} we show  zooms in on the parameter planes for several $n$. The ranges of $\alpha$  include the bifurcation parameters $\alpha=0$ (Lemma~\ref{solofijos}), $\alpha=1/2$ (Lemma~\ref{lemmahalley}) and  $\alpha_{\pm} =\frac{1\pm \sqrt{2\left(
1-n\right) }}{n-1}$ (Lemma~\ref{lemmaprecrit}). The parameter $\alpha=1/2$ (Halley's method) appears in the figures as a tip of an antenna which joins it with $\alpha=0$ (Chebyshev's method). We observe that the bifurcation structure at $\alpha=1/2$ is rather simple, which could indicate stability of the family for parameters near  Halley's  value. On the other hand, a more complex bifurcation structure appears at Chebyshev's parameter ($\alpha=0$). Recall that for this parameter all critical points collide at $z=0$, which is a preimage of the fixed point $z=\infty$. Indeed, for $n=25$ we observe a black disk of parameters for which the orbit of the critical point has not converged to a root after 150 iterates.
 This black disk does not correspond to any stable behaviour, but to parameters for which the critical orbit would require more iterates to converge to a root. Indeed,  this disk of parameters decreases when we increase the number of iterations. A similar situation occurs near the parameters $\alpha_{\pm}$, which are marked with  white points in Figure~\ref{fig:paramzoombif}. These parameters correspond to operators for which the critical points are preimages of $z=0$.
  Several of these black regions not corresponding to stable behaviour appear when drawing the parameter planes for big $n$, being the bigger ones around $\alpha_{\pm}$. Numerical experiments seem to indicate that the biggest black regions not related to stable behaviour appear around parameters for which the free critical points are eventually mapped under iteration of the corresponding operator onto $z=0$. In Section~\ref{sec_Numerico_dynam} we show the dynamics of $O_{n,\alpha}$ for all the bifurcation parameters discussed in this paragraph in order to analyse if they have to be avoided or if they present good dynamical behaviour.

\begin{figure}[hbt!]
\centering
\subfigure[\small{$n=10$} ]{
    \begin{tikzpicture}
    \begin{axis}[width=200pt, axis equal image, scale only axis,  enlargelimits=false, axis on top, %xtick={-3,-2,-1,0,1,2,3}, ytick={-3,-2,-1,0,1,2,3}
    ]
      \addplot graphics[xmin=1,xmax=1.10,ymin=-0.05,ymax=0.05] {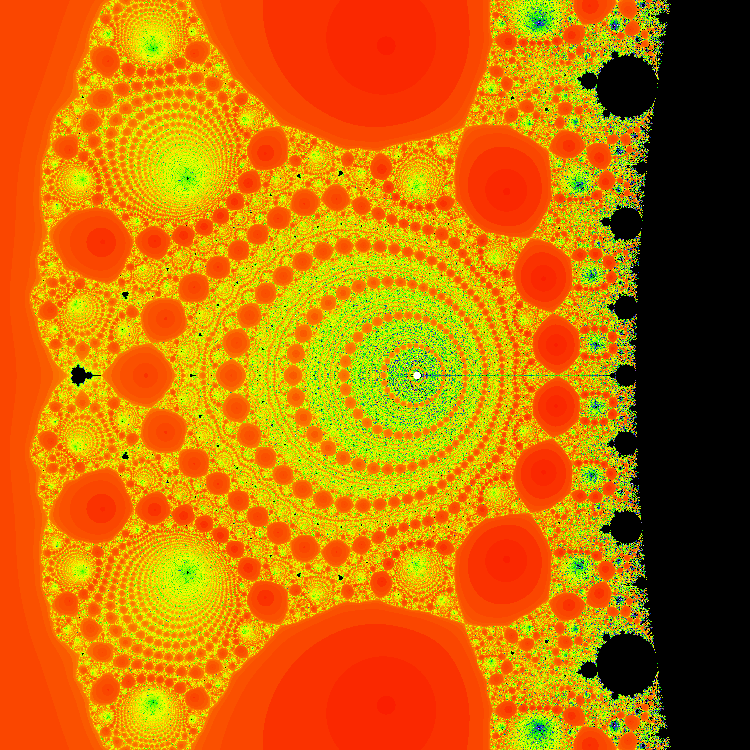};
    \end{axis}
  \end{tikzpicture}
    }
  \subfigure[\small{$n=25$} ]{
    \begin{tikzpicture}
    \begin{axis}[width=200pt, axis equal image, scale only axis,  enlargelimits=false, axis on top, %xtick={-3,-2,-1,0,1,2,3}, ytick={-3,-2,-1,0,1,2,3}
    ]
      \addplot graphics[xmin=0.95,xmax=1.05,ymin=-0.05,ymax=0.05]{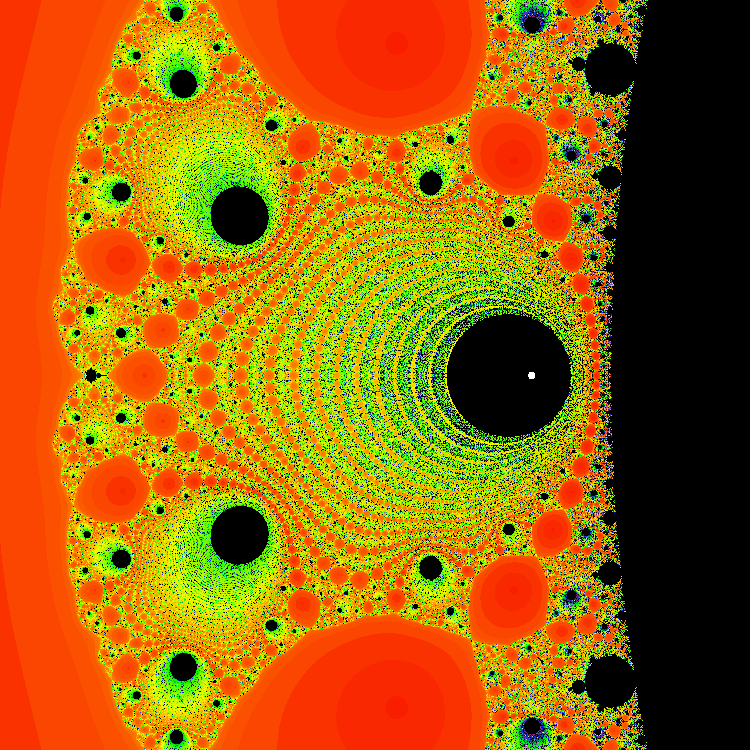};
    \end{axis}
  \end{tikzpicture}
    }
      \caption{\small Zooms in on the parameter planes near the degeneracy parameter $\alpha=\frac{2n-1}{2n-2}$.\label{fig:paramzoomsing}}
    \end{figure}

In Figure~\ref{fig:paramzoomsing} we show zooms in on the parameter plane near the parameter $\alpha=\frac{2n-1}{2n-2}$. If $n>2$, this bifurcation parameter corresponds to a lower degree operator for which $\{0,\infty\}$ is a superattracting cycle. Therefore, this is a parameter to be avoided. In Figure~\ref{fig:paramzoomsing} we indicate it with a small white point.
We observe how, for $n=10$ and $n=25$, this parameter is the center of a cascade of bifurcations. Similarly to what happens in the parameter plane near the other bifurcation parameters, this cascade of bifurcations may lead to a black region which does not correspond to stable behaviour (see Figure~\ref{fig:paramzoomsing} (b)). This cascade of bifurcations  takes place in a rather small region which is adjacent to the hyperbolic component of parameters which contains $\alpha=\frac{2n-1}{3n-3}.$ For fixed $n$, this parameter corresponds to the single operator $O_{n,\alpha}$ with order of convergence 4 to the roots (see Proposition~\ref{grado4}). Moreover, the parameters within this hyperbolic component have connected Julia set (see Corollary~\ref{corolario_collar}), which a priori makes of them desirable parameters. However, the hyperbolic component seems to become smaller and to move slightly to the left when we increase $n$ (see Figure~\ref{fig:paramgaton}). In particular, the parameter $\alpha=1$, which has order of convergence 4 to the roots for $n=2$, falls into the cascade of bifurcations for $n=25$ (see Figure~\ref{fig:paramzoomsing}). In Section~\ref{sec_Numerico_dynam} we show how the dynamics of $\alpha=1$ evolves when we increase $n$ and it approaches and enters the cascade of bifurcations.

\subsection{Dynamical planes}\label{sec_Numerico_dynam}

In this section we provide numerical drawings of dynamical planes of the operators $O_{n,\alpha}$. The drawings of dynamical planes are done using a program written in C which works as follows. We take a grid of  $1500\times1500$ and we associate a point $z\in\com$ to each point of the grid. The range of the real and the imaginary part of the points $z$ is indicated in the horizontal and vertical axes of the images, respectively. Then, we iterate the point $z$ up to 75 times. At each iteration we verify if the iterated point $w$ has converged to any of the $n$th-roots of the unity (we verify if $|w-\xi|<10^{-4}$, for any $n$th-root of the unity $\xi$). If it converges to any root of the unity, we colour the corresponding point in the grid with an scaling from red (fast convergence) to yellow,  green, blue, purple and to grey (slow convergence). If after 75 iterations the orbit has not converged to a root, we plot the point in black.

\begin{figure}[hbt!]
\centering
\subfigure[\small{$n=3, \alpha=0.2+1.4i$} ]{
    \begin{tikzpicture}
    \begin{axis}[width=195pt, axis equal image, scale only axis,  enlargelimits=false, axis on top, xtick={-3,-2,-1,0,1,2,3}, ytick={-3,-2,-1,0,1,2,3} ]
      \addplot graphics[xmin=-3,xmax=3,ymin=-3,ymax=3] {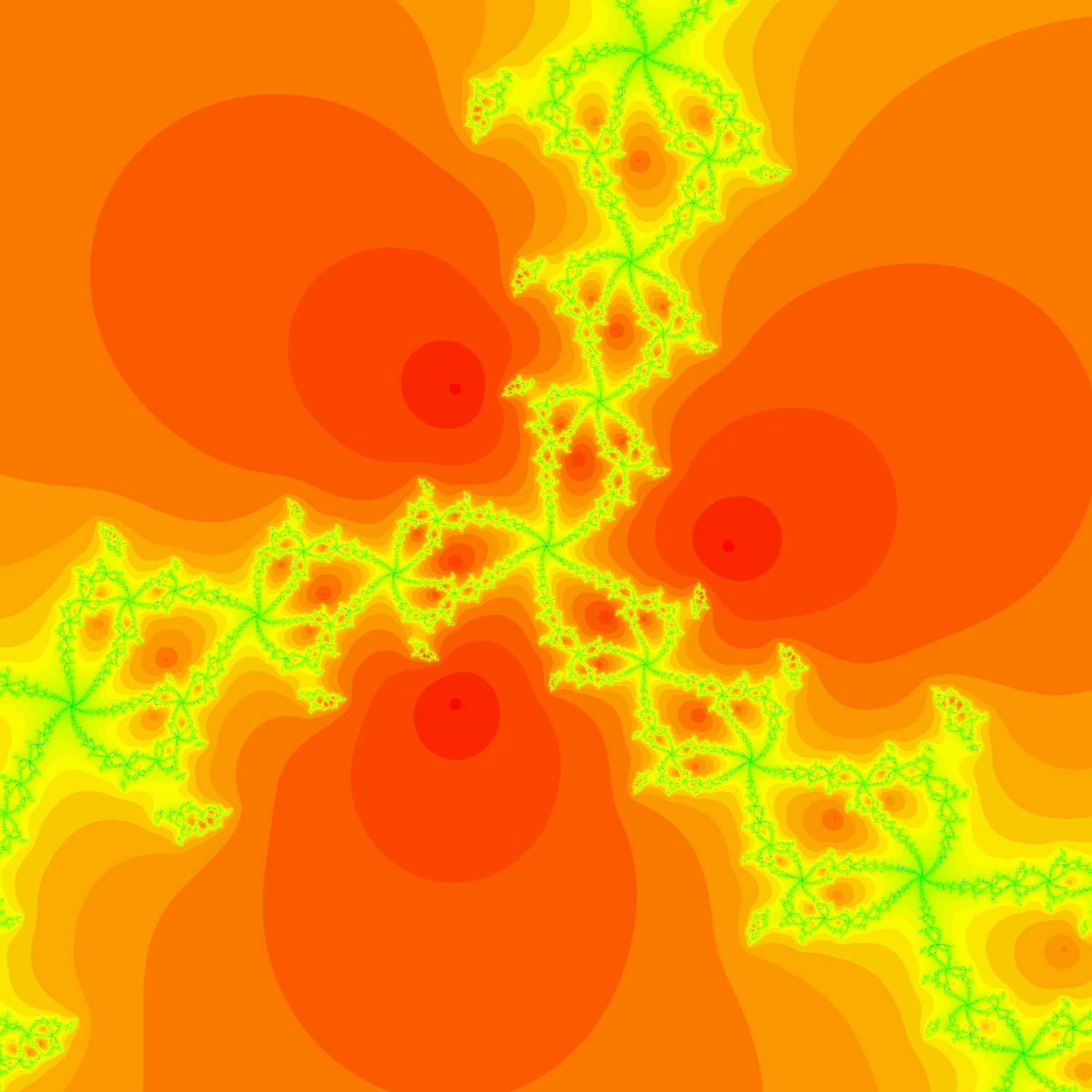};
    \end{axis}
  \end{tikzpicture}
    }
  \subfigure[\small{$n=3, \alpha=2i$} ]{
    \begin{tikzpicture}
    \begin{axis}[width=195pt, axis equal image, scale only axis,  enlargelimits=false, axis on top, xtick={-3,-2,-1,0,1,2,3}, ytick={-3,-2,-1,0,1,2,3} ]
      \addplot graphics[xmin=-3,xmax=3,ymin=-3,ymax=3] {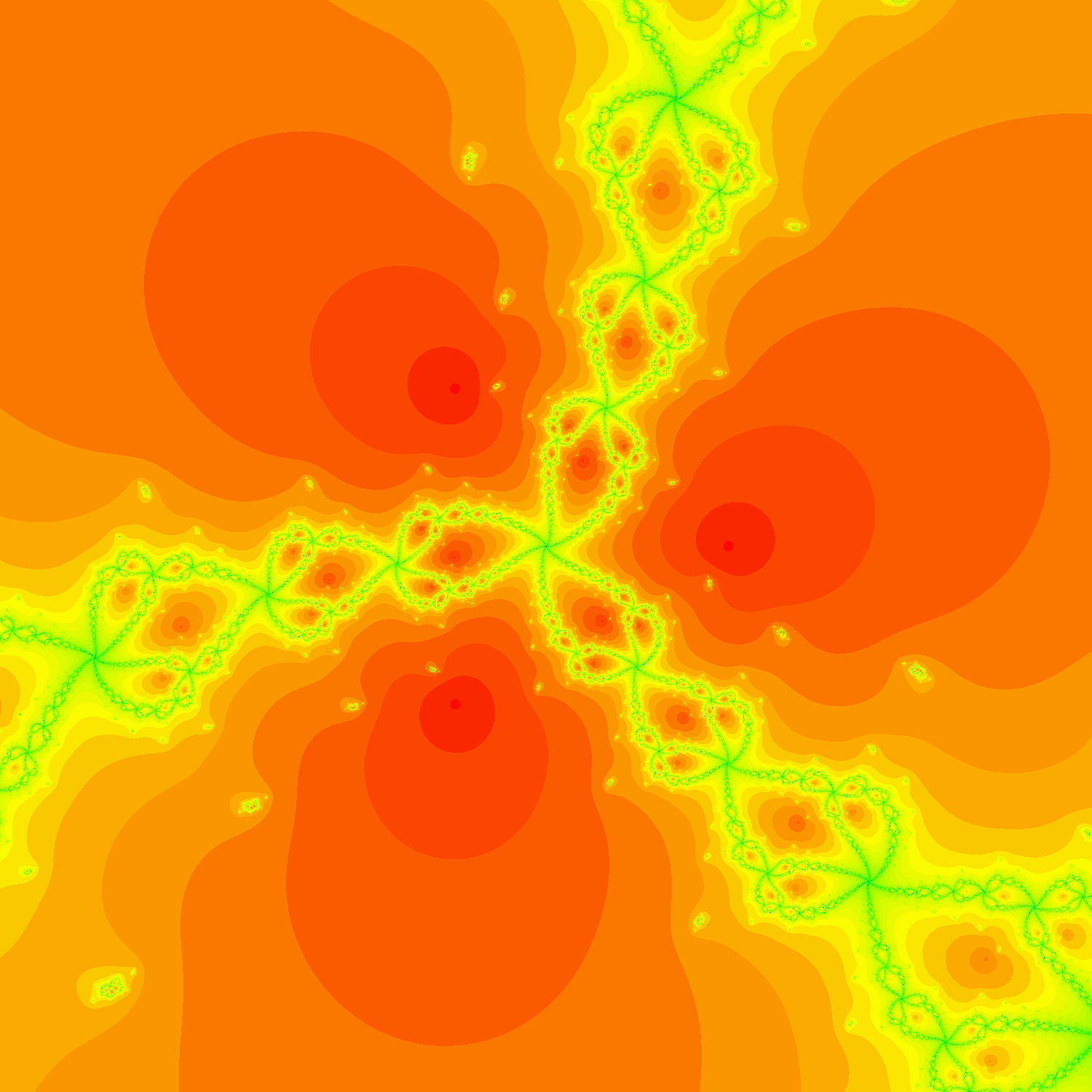};
    \end{axis}
  \end{tikzpicture}
    }
  \subfigure[\small{$n=25, \alpha=0.2+1.4i$} ]{
    \begin{tikzpicture}
    \begin{axis}[width=195pt, axis equal image, scale only axis,  enlargelimits=false, axis on top, xtick={-3,-2,-1,0,1,2,3}, ytick={-3,-2,-1,0,1,2,3} ]
      \addplot graphics[xmin=-3,xmax=3,ymin=-3,ymax=3] {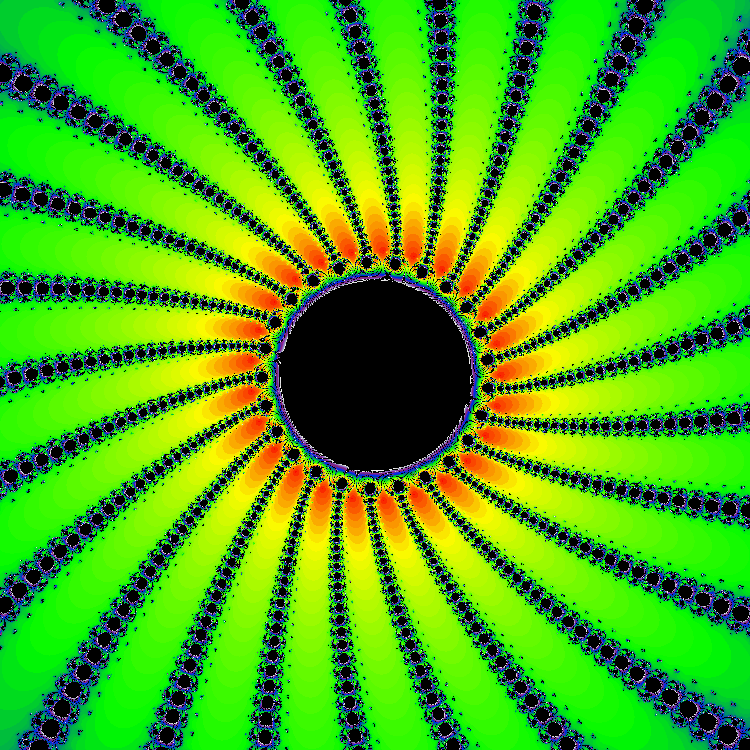};
    \end{axis}
  \end{tikzpicture}
    }
    \subfigure[\small{$n=25, \alpha=2i$}]{
    	\begin{tikzpicture}
    		\begin{axis}[width=195pt, axis equal image, scale only axis,  enlargelimits=false, axis on top, xtick={-3,-2,-1,0,1,2,3}, ytick={-3,-2,-1,0,1,2,3} ]
      			\addplot graphics[xmin=-3,xmax=3,ymin=-3,ymax=3] {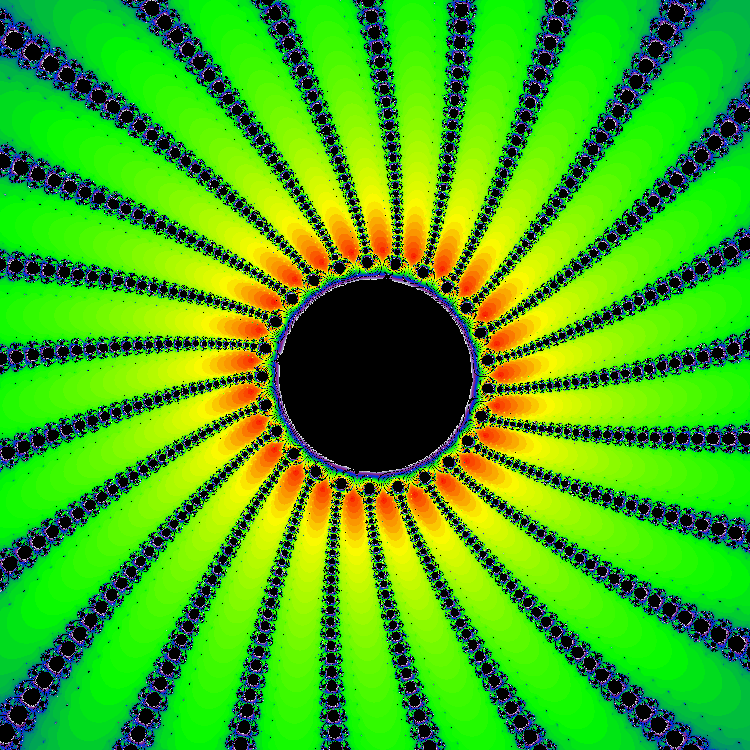};
    		\end{axis}
  		\end{tikzpicture}	
  }
      \caption{\small Dynamical planes of several $O_{n, \alpha}(z)$ with a disconnected Julia set.\label{fig:dynamdiscon}}
    \end{figure}

The goal of this section is to analyse the dynamical behaviour of the operators $O_{n,\alpha}$ for several parameters in order to have a better idea of which parameters provide a good dynamical behaviour. We first investigate numerically for which parameters the Julia set is disconnected and we analyse their dynamics. Before doing so, we recall the concept of hyperbolic component of parameters. A hyperbolic component is a connected set of parameters for which all critical orbits accumulate on attracting or superattracting cycles. The Julia set is stable within hyperbolic components, i.e.\ the Julia set is connected for one parameter of the component if and only if it is connected for all parameters of the component. Hyperbolic components of parameters for which the free critical points belong to the basins of attraction of the roots appear in red in Figure~\ref{fig:paramgaton}.
 In Corollary~\ref{corolario_collar} we have proven that the operators corresponding to the hyperbolic component which contains $\alpha=\frac{2n-1}{3n-3}$ have connected Julia set. Numerical simulations seem to indicate that the other red hyperbolic components bounded by the Collar of the Cat set correspond to parameters for which the free critical points belong to the basins of attraction of the roots but not to their immediate basins of attraction and, hence, correspond to operators with a connected Julia set (see Theorem~\ref{thmconJulia}).
  The only hyperbolic component corresponding to operators with disconnected Julia set seems to be the red unbounded component on the complement of the Collar (see Figure~\ref{fig:paramgaton}). In Figure~\ref{fig:dynamdiscon}, Figure~\ref{fig:dynam3} (f), and Figure~\ref{fig:dynam25} (f) we show  the dynamical planes of several parameters within this hyperbolic component for $n=3$ and $n=25$. We take the parameter $\alpha=0.2+1.4i$, which is close to the Collar, and the parameters $\alpha=2i$ and $\alpha=4i$, which are farther away from the Collar. For $\alpha=0.2+1.4i$ the holes in the basins of attraction are relatively big and are easy to observe. For the parameters $\alpha=2i$ and $\alpha=4i$ these holes are much more difficult to see since they become smaller. However, the convergence to the roots within this last parameters is slower. Indeed, for $n=3$ we observe how the basins of attractions to the roots have a more intense red tonality for $\alpha=0.2+1.4i$, which indicates fast convergence. For $n=25$ we can observe how the basins of attraction of the roots acquire a blueish tonality for $\alpha=2i$ and, particularly, for $\alpha=4i$, which indicates slower convergence to the roots. We may conclude that, within this hyperbolic component, parameters close to the Collar have bigger holes on the basins of attraction of the roots, but the convergence to the roots is faster.

\begin{figure}[p]
\centering
\subfigure[\small{$\alpha=5/6$} ]{
    \begin{tikzpicture}
    \begin{axis}[width=195pt, axis equal image, scale only axis,  enlargelimits=false, axis on top, xtick={-3,-2,-1,0,1,2,3}, ytick={-3,-2,-1,0,1,2,3} ]
      \addplot graphics[xmin=-3,xmax=3,ymin=-3,ymax=3] {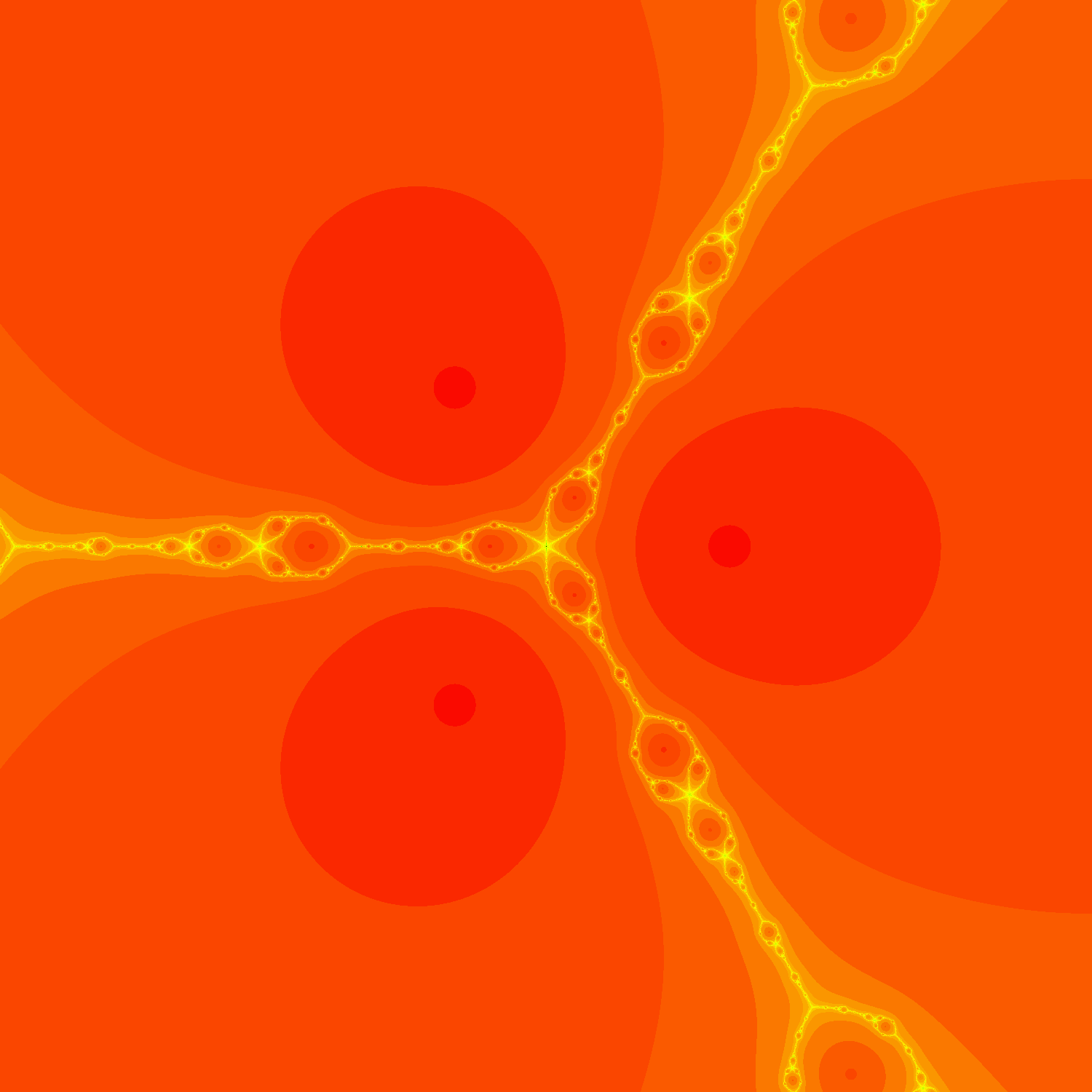};
    \end{axis}
  \end{tikzpicture}
    }
  \subfigure[\small{$\alpha=0.5$, Halley's method} ]{
    \begin{tikzpicture}
    \begin{axis}[width=195pt, axis equal image, scale only axis,  enlargelimits=false, axis on top, xtick={-3,-2,-1,0,1,2,3}, ytick={-3,-2,-1,0,1,2,3} ]
      \addplot graphics[xmin=-3,xmax=3,ymin=-3,ymax=3] {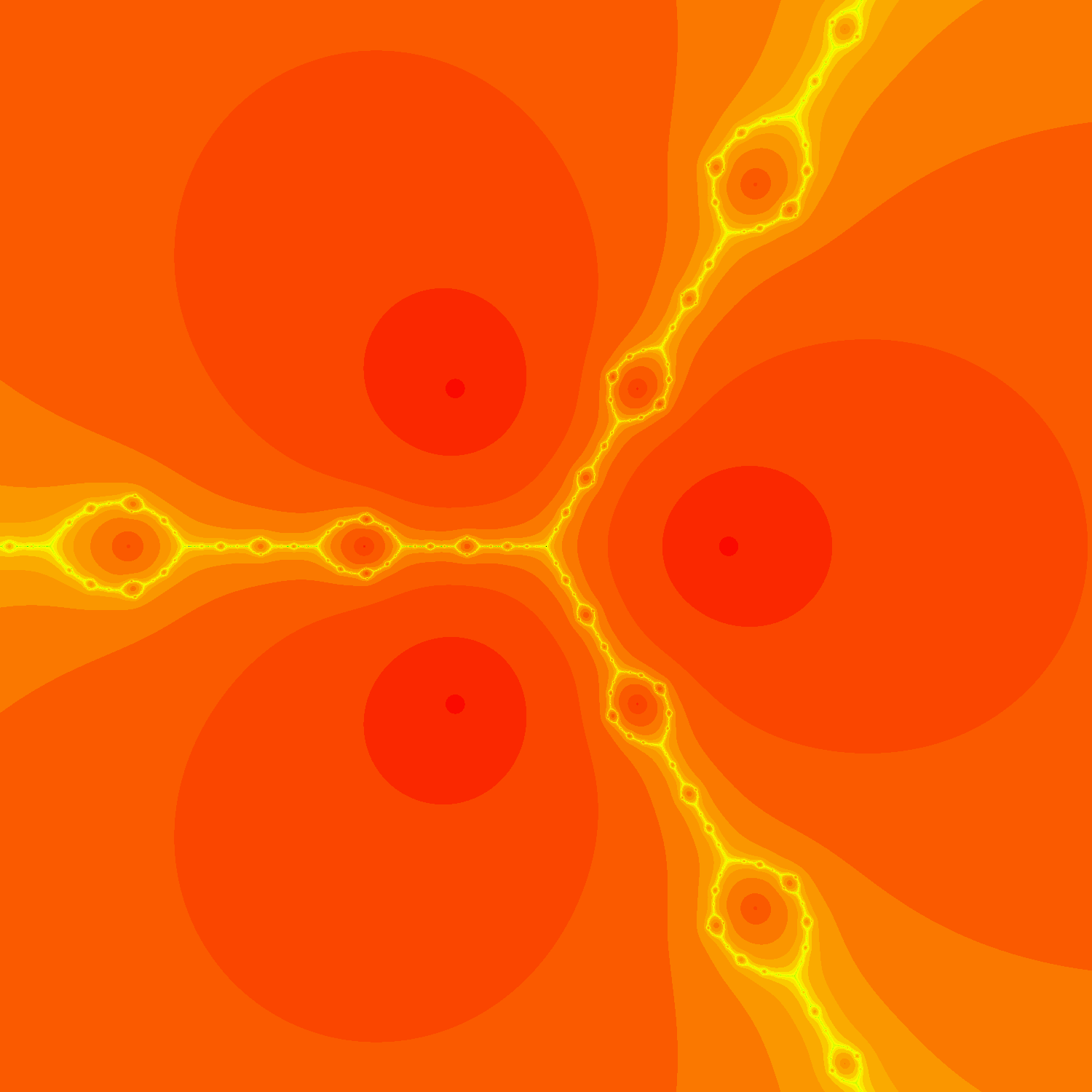};
    \end{axis}
  \end{tikzpicture}
    }
    \subfigure[\small{$\alpha=0$, Chebyshev's method}]{
    	\begin{tikzpicture}
    		\begin{axis}[width=195pt, axis equal image, scale only axis,  enlargelimits=false, axis on top, xtick={-3,-2,-1,0,1,2,3}, ytick={-3,-2,-1,0,1,2,3} ]
      			\addplot graphics[xmin=-3,xmax=3,ymin=-3,ymax=3] {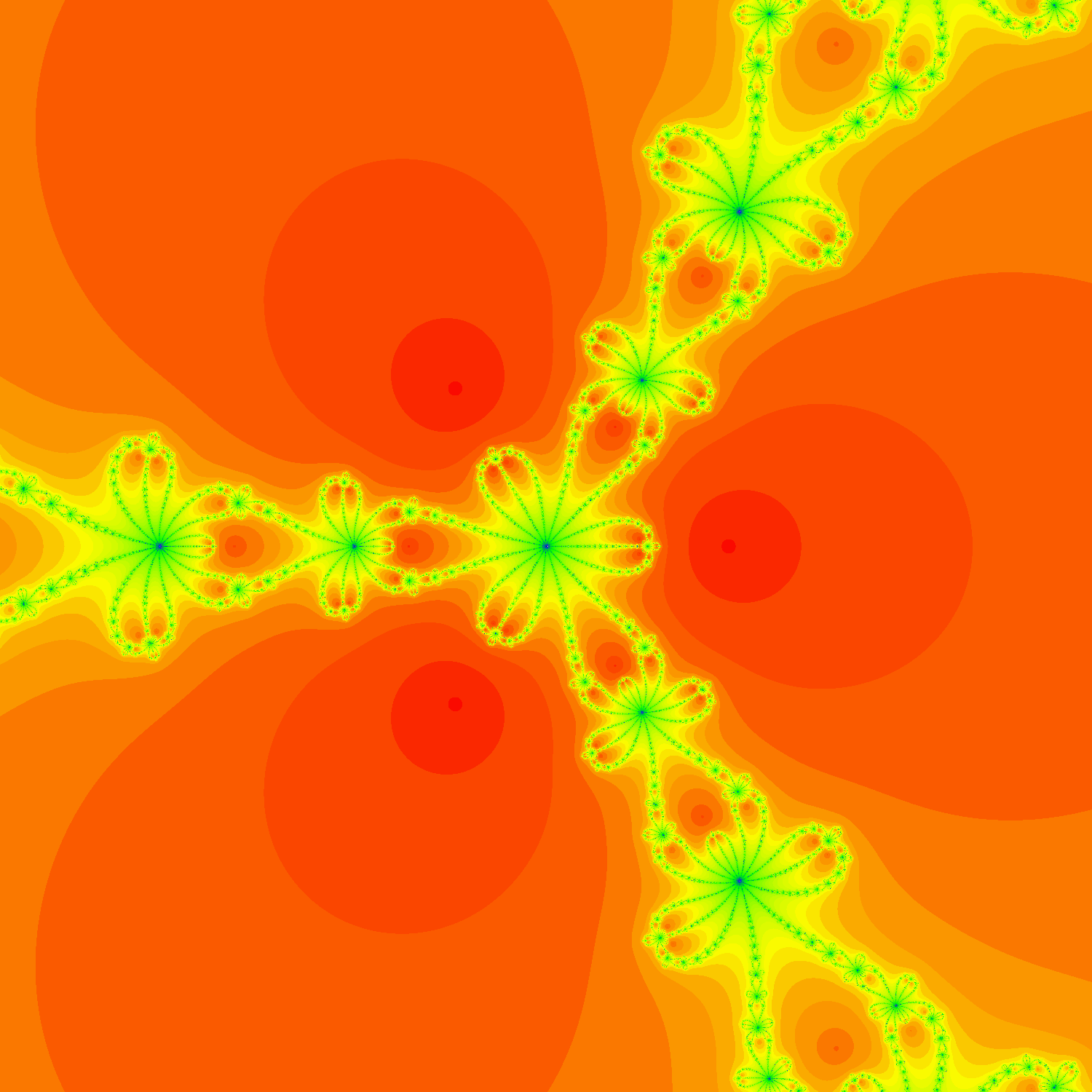};
    		\end{axis}
  		\end{tikzpicture}	
  }
  \subfigure[\small{$\alpha=1$, Super-Halley method} ]{
    \begin{tikzpicture}
    \begin{axis}[width=195pt, axis equal image, scale only axis,  enlargelimits=false, axis on top, xtick={-3,-2,-1,0,1,2,3}, ytick={-3,-2,-1,0,1,2,3} ]
      \addplot graphics[xmin=-3,xmax=3,ymin=-3,ymax=3] {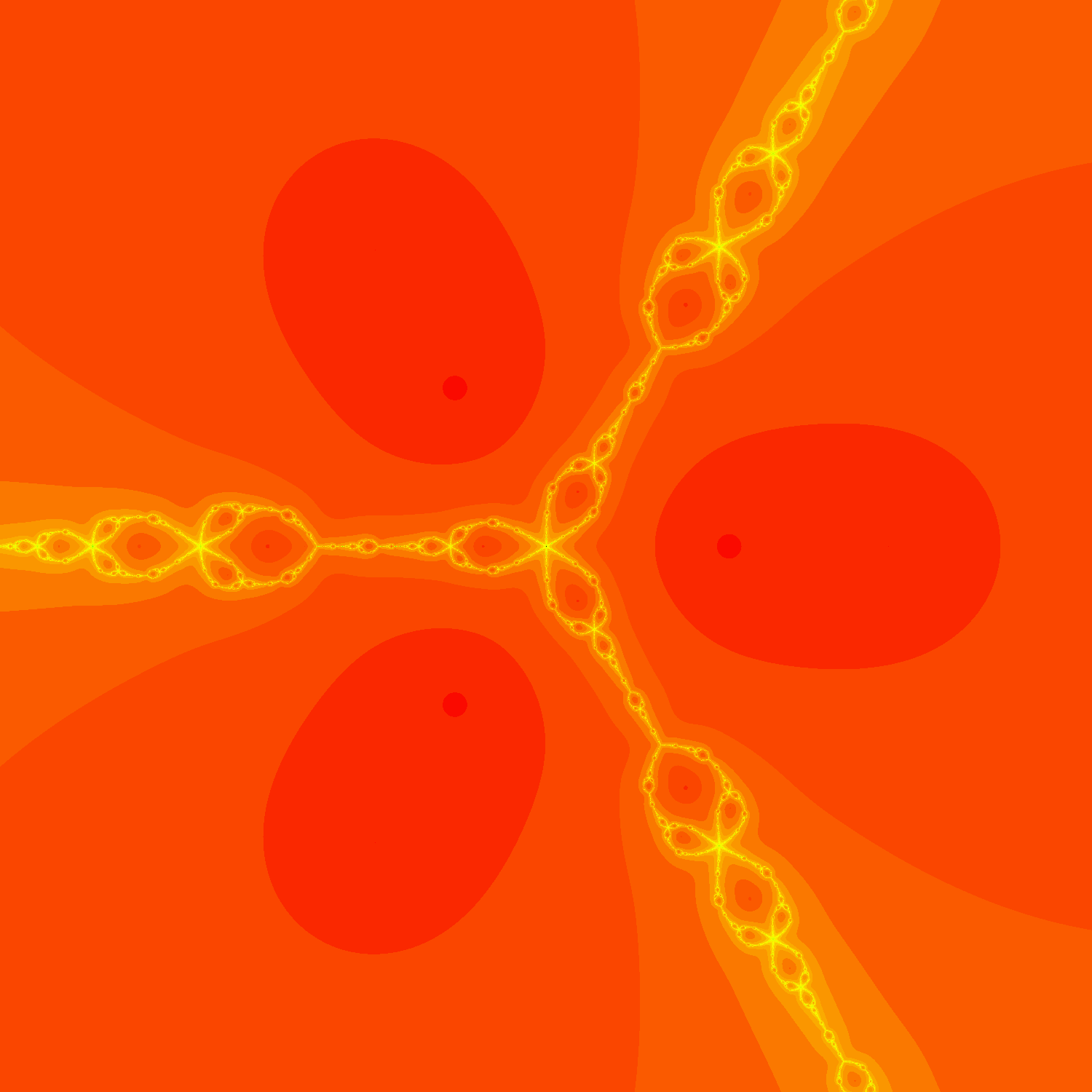};
    \end{axis}
  \end{tikzpicture}
    }
    \subfigure[\small{$\alpha=(1+i\sqrt{4})/2$}]{
    	\begin{tikzpicture}
    		\begin{axis}[width=195pt, axis equal image, scale only axis,  enlargelimits=false, axis on top, xtick={-3,-2,-1,0,1,2,3}, ytick={-3,-2,-1,0,1,2,3} ]
      			\addplot graphics[xmin=-3,xmax=3,ymin=-3,ymax=3] {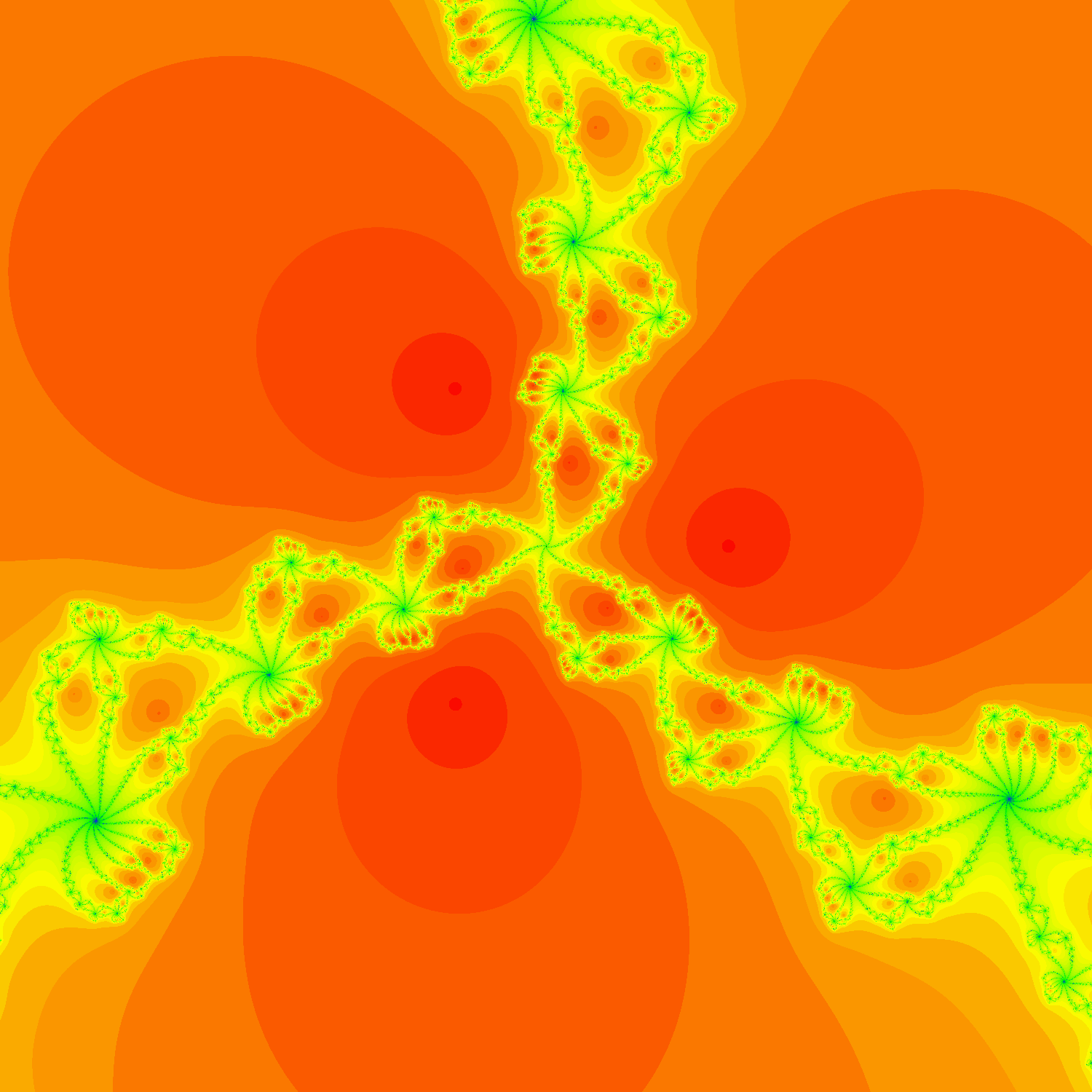};
    		\end{axis}
  		\end{tikzpicture}	
  }
  \subfigure[\small{$\alpha=4i$} ]{
    \begin{tikzpicture}
    \begin{axis}[width=195pt, axis equal image, scale only axis,  enlargelimits=false, axis on top,  xtick={-3,-2,-1,0,1,2,3}, ytick={-3,-2,-1,0,1,2,3}  ]
      \addplot  graphics[xmin=-3,xmax=3,ymin=-3,ymax=3] {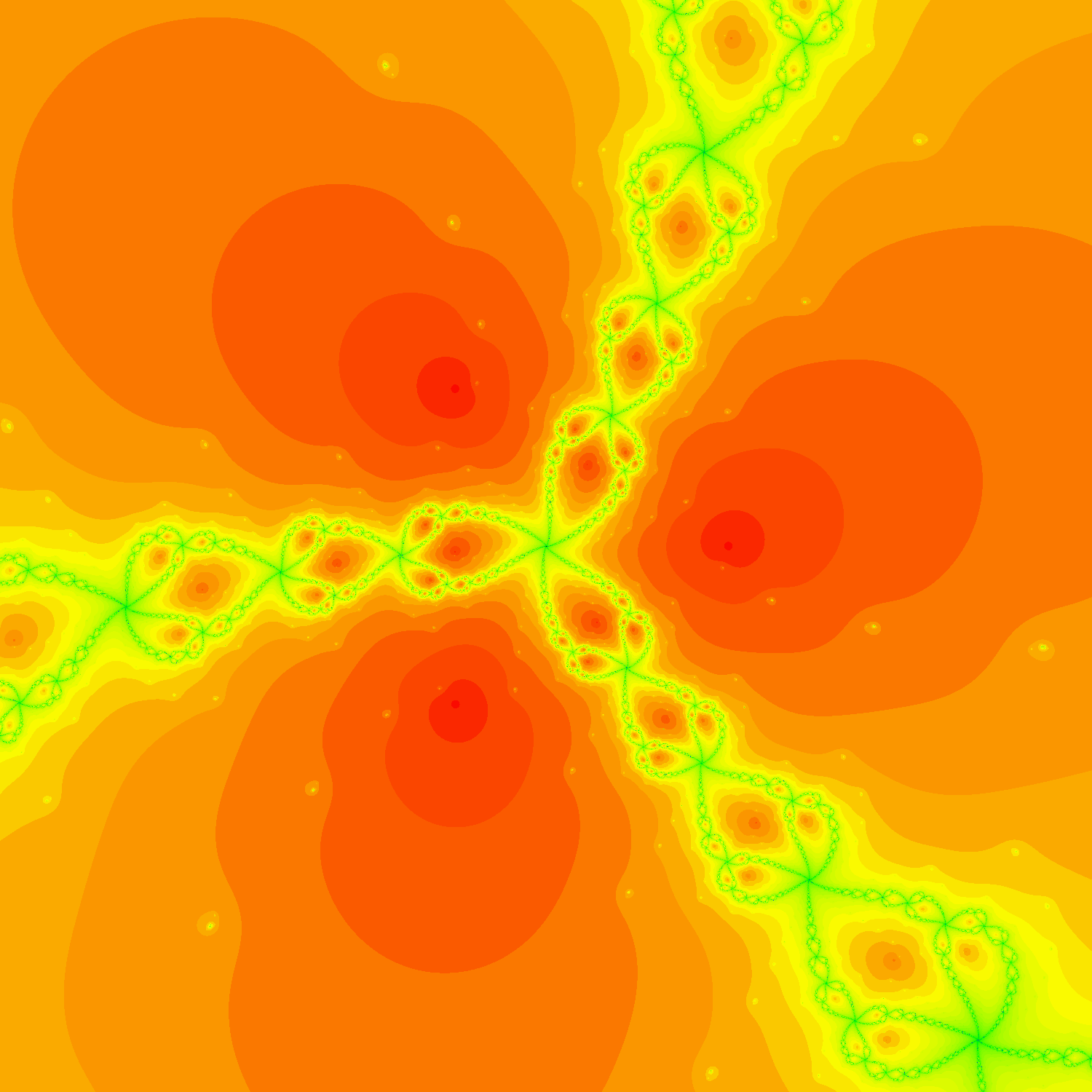};
    \end{axis}
  \end{tikzpicture}
    }
      \caption{\small Dynamical planes of $O_{n, \alpha}(z)$ for  $n=3$ and different values of $\alpha$. \label{fig:dynam3}}
    \end{figure}

\begin{figure}[p]
\centering
\subfigure[\small{$\alpha=19/27$} ]{
    \begin{tikzpicture}
    \begin{axis}[width=195pt, axis equal image, scale only axis,  enlargelimits=false, axis on top, xtick={-3,-2,-1,0,1,2,3}, ytick={-3,-2,-1,0,1,2,3} ]
      \addplot graphics[xmin=-3,xmax=3,ymin=-3,ymax=3] {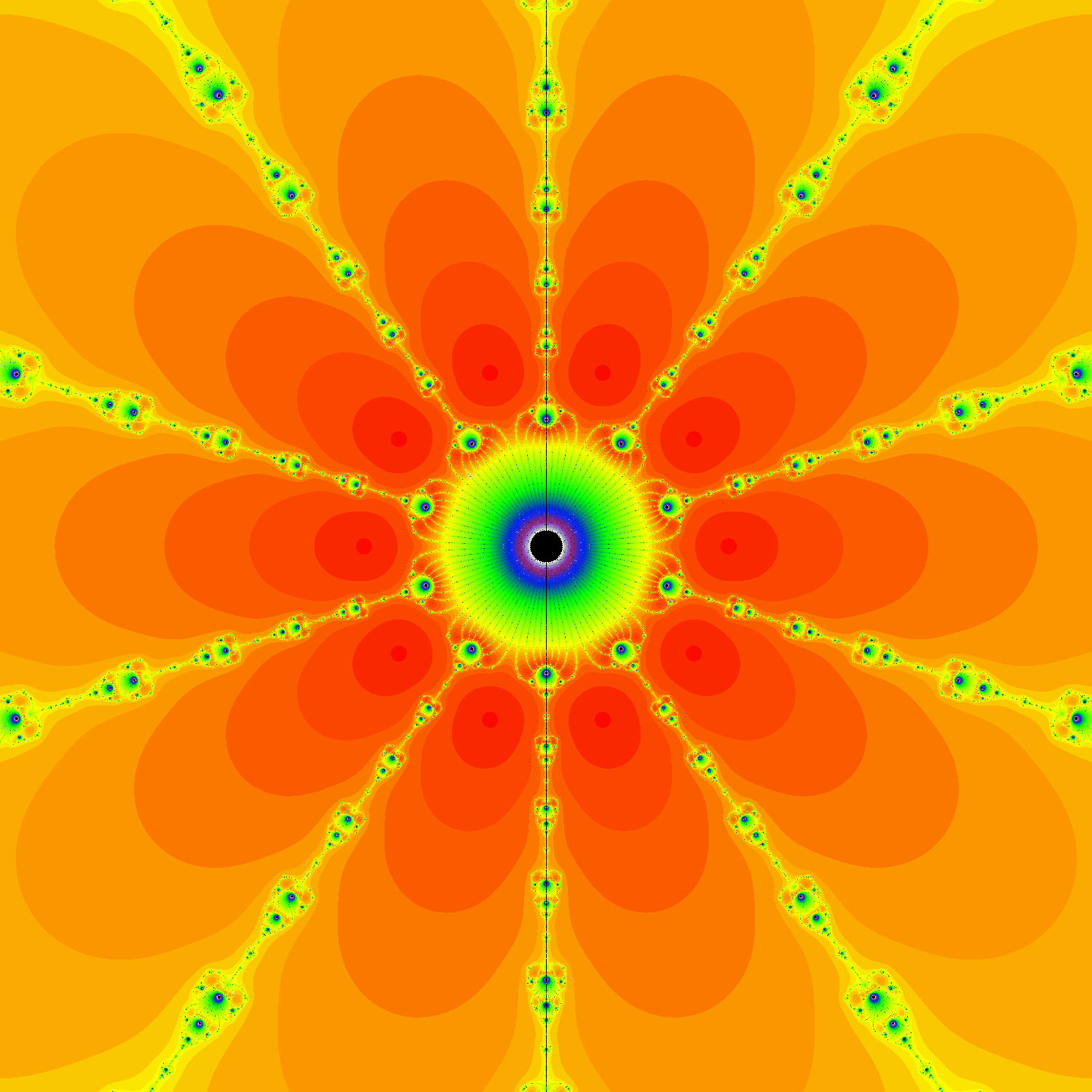};
    \end{axis}
  \end{tikzpicture}
    }
  \subfigure[\small{$\alpha=0.5$, Halley's method} ]{
    \begin{tikzpicture}
    \begin{axis}[width=195pt, axis equal image, scale only axis,  enlargelimits=false, axis on top, xtick={-3,-2,-1,0,1,2,3}, ytick={-3,-2,-1,0,1,2,3} ]
      \addplot graphics[xmin=-3,xmax=3,ymin=-3,ymax=3] {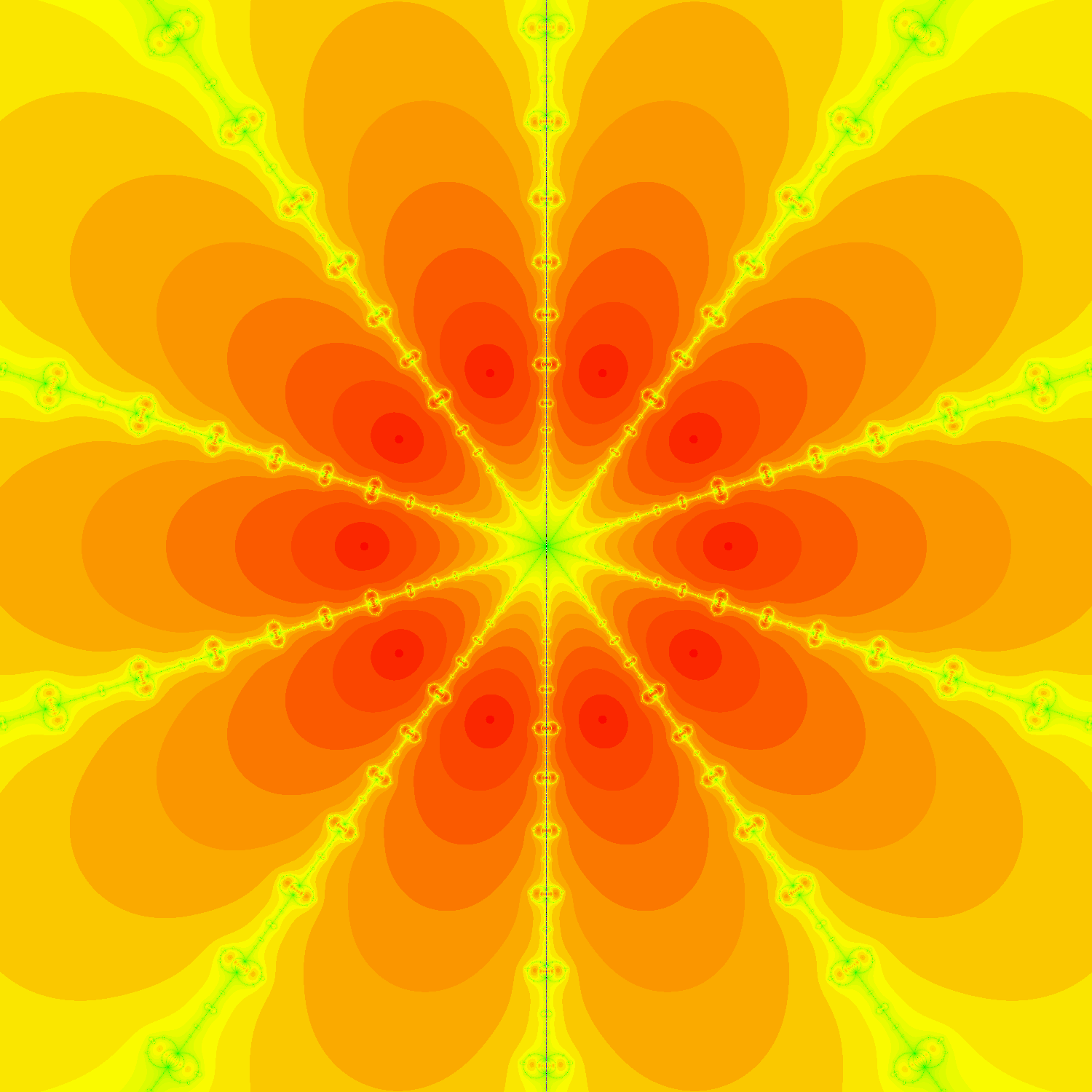};
    \end{axis}
  \end{tikzpicture}
    }
    \subfigure[\small{$\alpha=0$, Chebyshev's method}]{
    	\begin{tikzpicture}
    		\begin{axis}[width=195pt, axis equal image, scale only axis,  enlargelimits=false, axis on top, xtick={-3,-2,-1,0,1,2,3}, ytick={-3,-2,-1,0,1,2,3} ]
      			\addplot graphics[xmin=-3,xmax=3,ymin=-3,ymax=3] {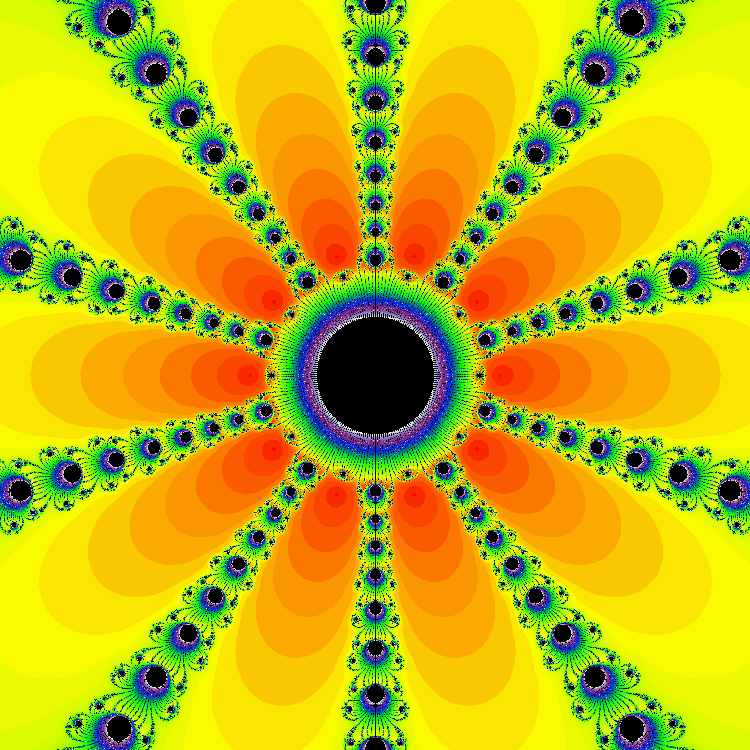};
    		\end{axis}
  		\end{tikzpicture}	
  }
  \subfigure[\small{$\alpha=1$, Super-Halley method} ]{
    \begin{tikzpicture}
    \begin{axis}[width=195pt, axis equal image, scale only axis,  enlargelimits=false, axis on top, xtick={-3,-2,-1,0,1,2,3}, ytick={-3,-2,-1,0,1,2,3} ]
      \addplot graphics[xmin=-3,xmax=3,ymin=-3,ymax=3] {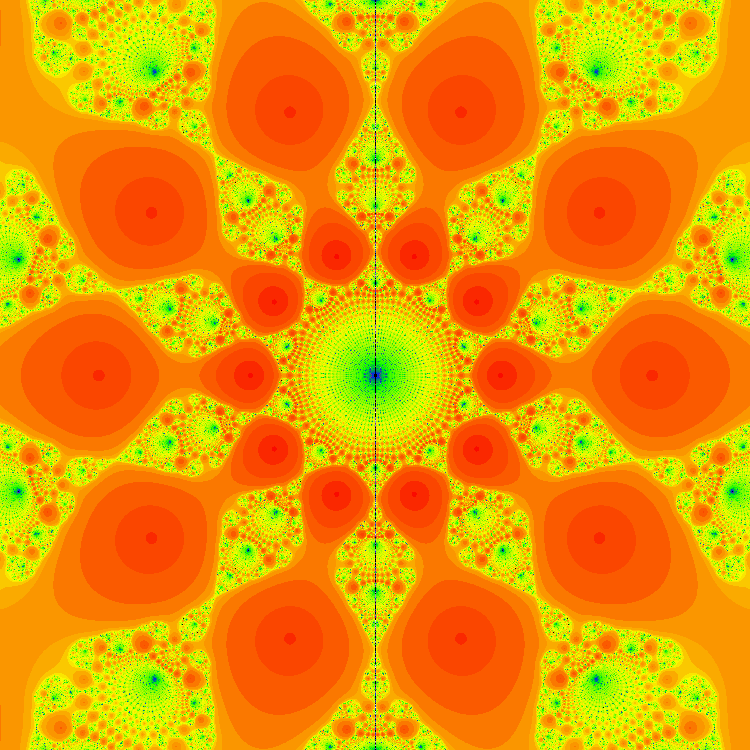};
    \end{axis}
  \end{tikzpicture}
    }
    \subfigure[\small{$\alpha=(1+i\sqrt{18})/9$}]{
    	\begin{tikzpicture}
    		\begin{axis}[width=195pt, axis equal image, scale only axis,  enlargelimits=false, axis on top, xtick={-3,-2,-1,0,1,2,3}, ytick={-3,-2,-1,0,1,2,3} ]
      			\addplot graphics[xmin=-3,xmax=3,ymin=-3,ymax=3] {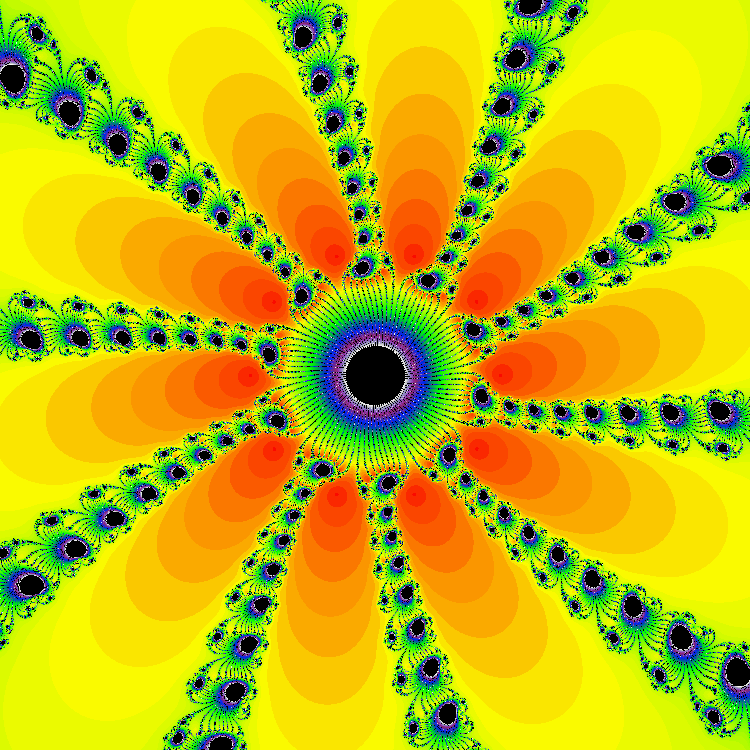};
    		\end{axis}
  		\end{tikzpicture}	
  }
  \subfigure[\small{$\alpha=4i$} ]{
    \begin{tikzpicture}
    \begin{axis}[width=195pt, axis equal image, scale only axis,  enlargelimits=false, axis on top,  xtick={-3,-2,-1,0,1,2,3}, ytick={-3,-2,-1,0,1,2,3}  ]
      \addplot  graphics[xmin=-3,xmax=3,ymin=-3,ymax=3] {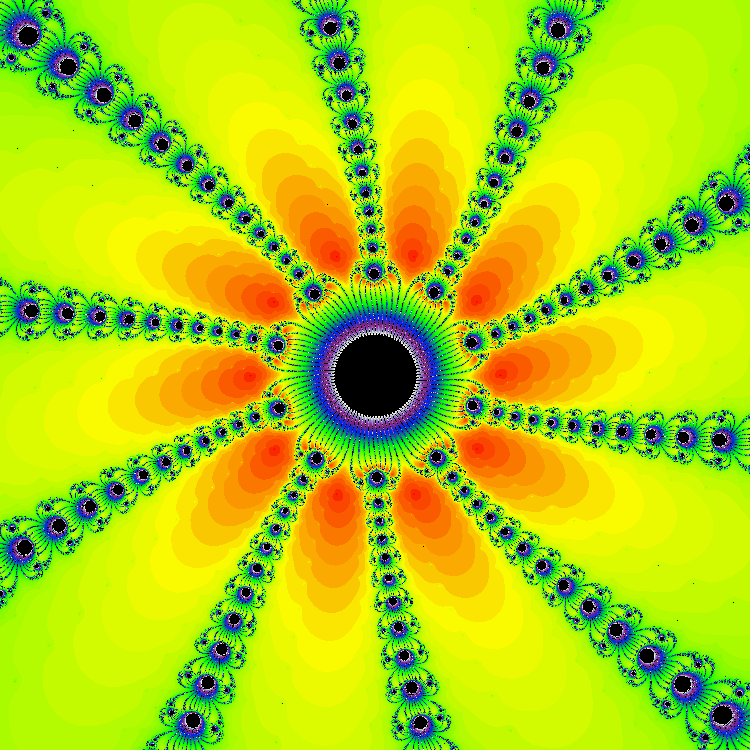};
    \end{axis}
  \end{tikzpicture}
    }
      \caption{\small Dynamical planes of $O_{n, \alpha}(z)$ for  $n=10$ and different values of $\alpha$.\label{fig:dynam10}}
    \end{figure}

\begin{figure}[p]
\centering
\subfigure[\small{$\alpha=49/72$} ]{
    \begin{tikzpicture}
    \begin{axis}[width=195pt, axis equal image, scale only axis,  enlargelimits=false, axis on top, xtick={-3,-2,-1,0,1,2,3}, ytick={-3,-2,-1,0,1,2,3} ]
      \addplot graphics[xmin=-3,xmax=3,ymin=-3,ymax=3] {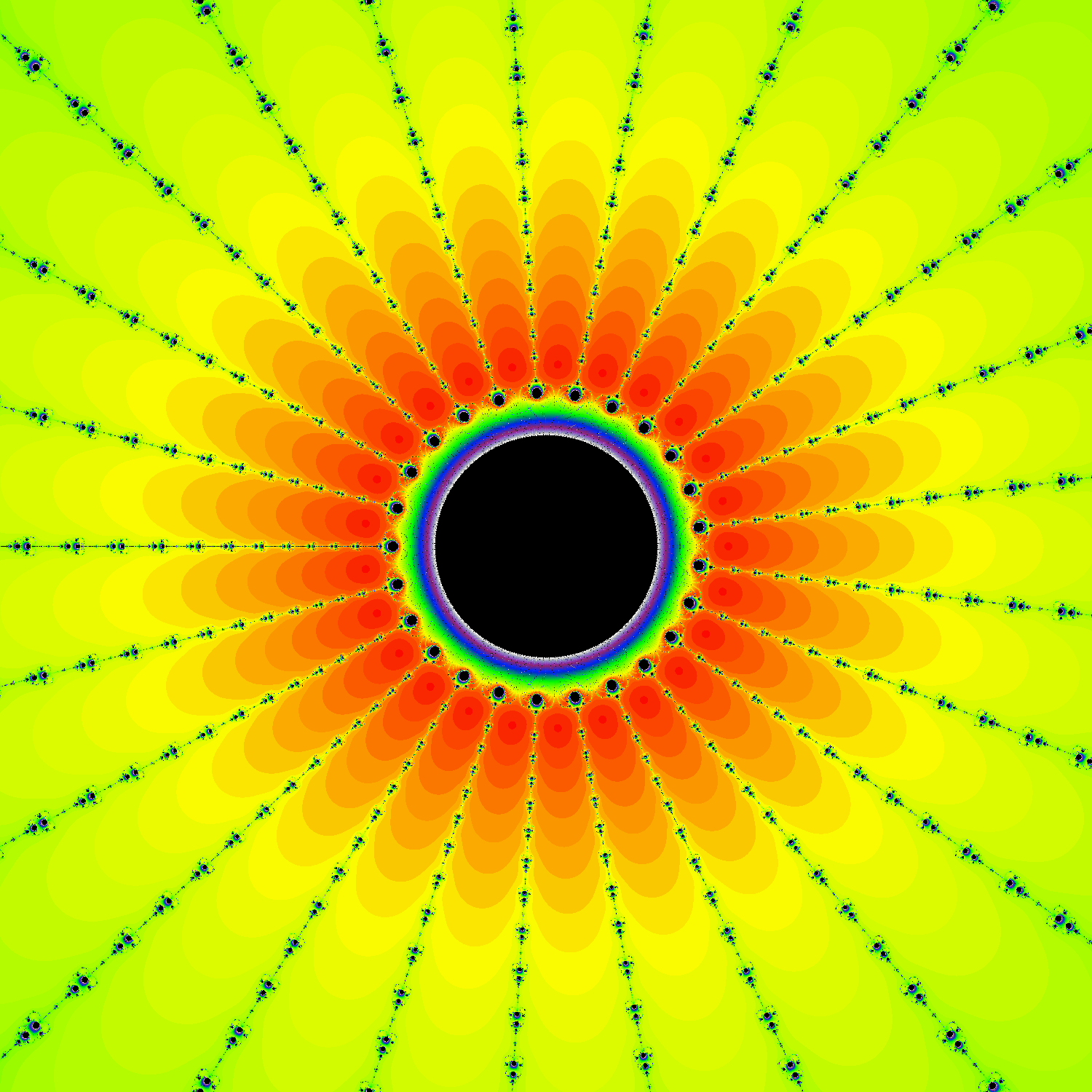};
    \end{axis}
  \end{tikzpicture}
    }
  \subfigure[\small{$\alpha=0.5$, Halley's method} ]{
    \begin{tikzpicture}
    \begin{axis}[width=195pt, axis equal image, scale only axis,  enlargelimits=false, axis on top, xtick={-3,-2,-1,0,1,2,3}, ytick={-3,-2,-1,0,1,2,3} ]
      \addplot graphics[xmin=-3,xmax=3,ymin=-3,ymax=3] {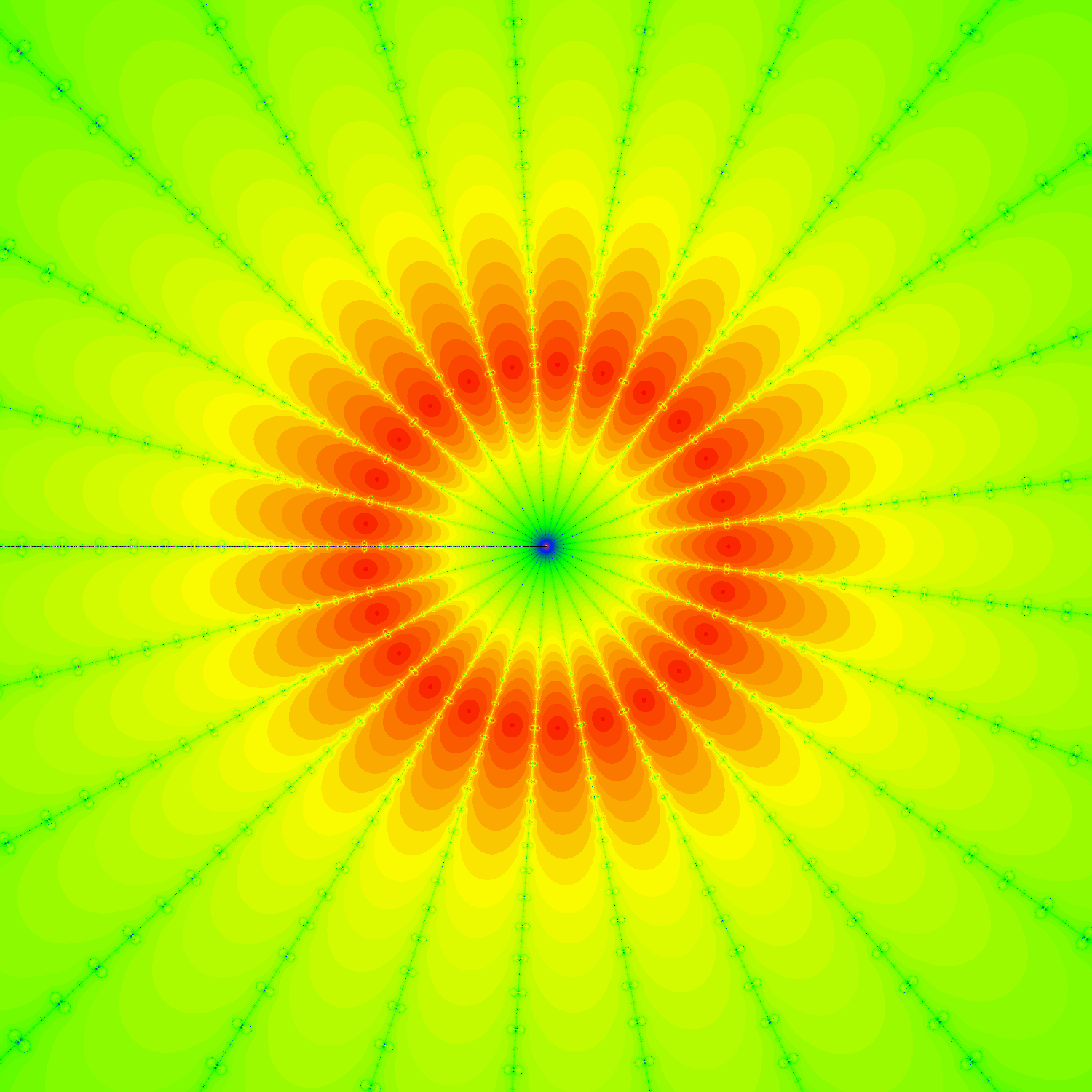};
    \end{axis}
  \end{tikzpicture}
    }
    \subfigure[\small{$\alpha=0$, Chebyshev's method}]{
    	\begin{tikzpicture}
    		\begin{axis}[width=195pt, axis equal image, scale only axis,  enlargelimits=false, axis on top, xtick={-3,-2,-1,0,1,2,3}, ytick={-3,-2,-1,0,1,2,3} ]
      			\addplot graphics[xmin=-3,xmax=3,ymin=-3,ymax=3] {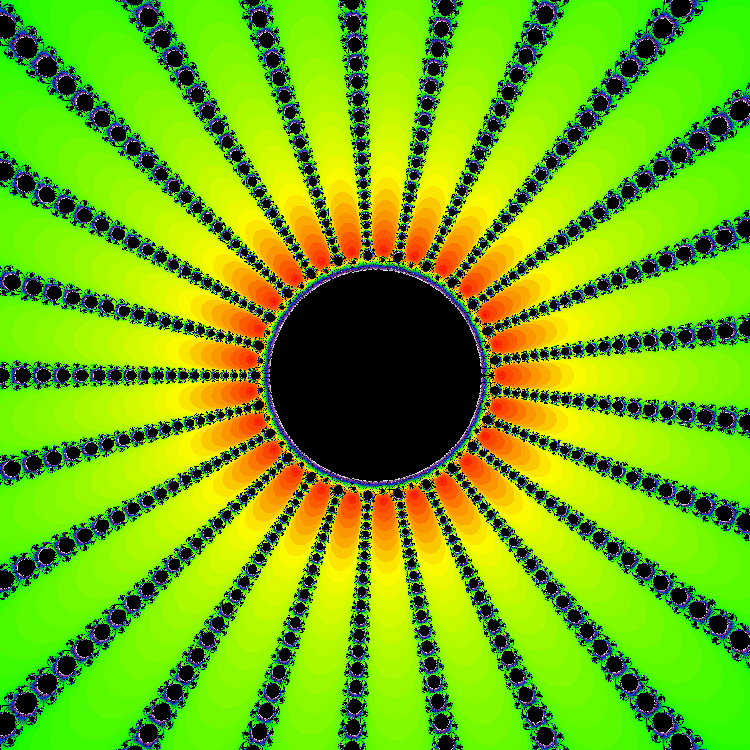};
    		\end{axis}
  		\end{tikzpicture}	
  }
  \subfigure[\small{$\alpha=1$, Super-Halley method} ]{
    \begin{tikzpicture}
    \begin{axis}[width=195pt, axis equal image, scale only axis,  enlargelimits=false, axis on top, xtick={-3,-2,-1,0,1,2,3}, ytick={-3,-2,-1,0,1,2,3} ]
      \addplot graphics[xmin=-3,xmax=3,ymin=-3,ymax=3] {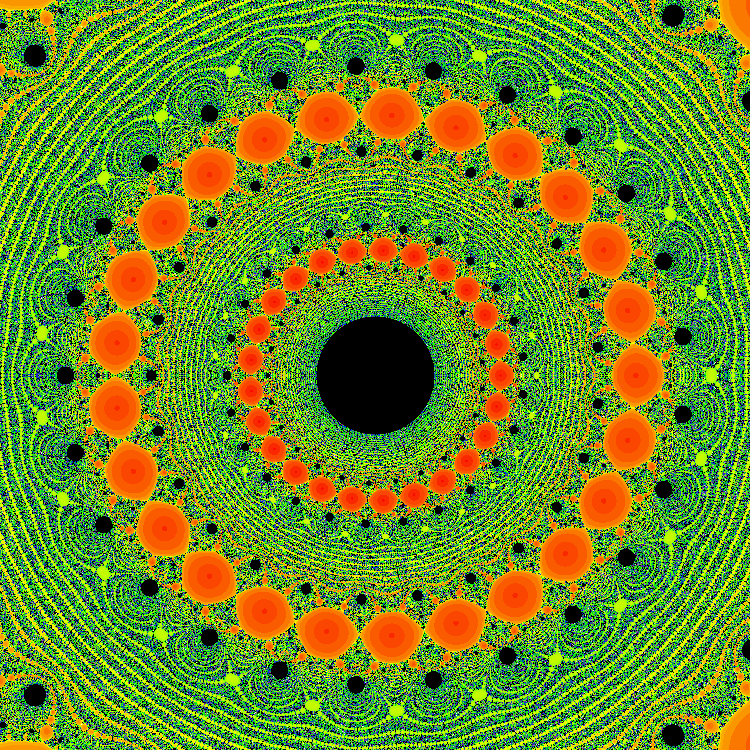};
    \end{axis}
  \end{tikzpicture}
    }
    \subfigure[\small{$\alpha=(1+i\sqrt{48})/24$}]{
    	\begin{tikzpicture}
    		\begin{axis}[width=195pt, axis equal image, scale only axis,  enlargelimits=false, axis on top, xtick={-3,-2,-1,0,1,2,3}, ytick={-3,-2,-1,0,1,2,3} ]
      			\addplot graphics[xmin=-3,xmax=3,ymin=-3,ymax=3] {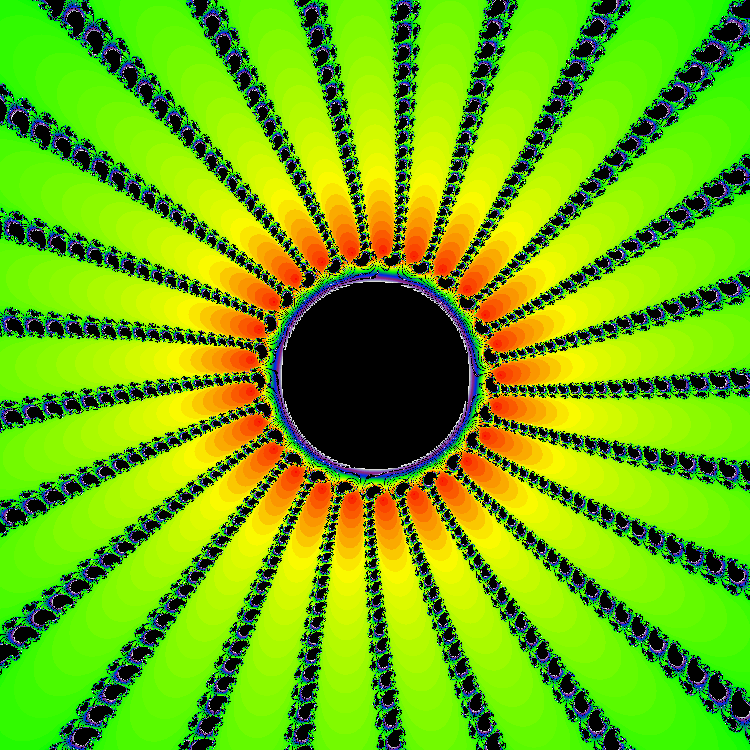};
    		\end{axis}
  		\end{tikzpicture}	
  }
  \subfigure[\small{$\alpha=4i$} ]{
    \begin{tikzpicture}
    \begin{axis}[width=195pt, axis equal image, scale only axis,  enlargelimits=false, axis on top,  xtick={-3,-2,-1,0,1,2,3}, ytick={-3,-2,-1,0,1,2,3}  ]
      \addplot  graphics[xmin=-3,xmax=3,ymin=-3,ymax=3] {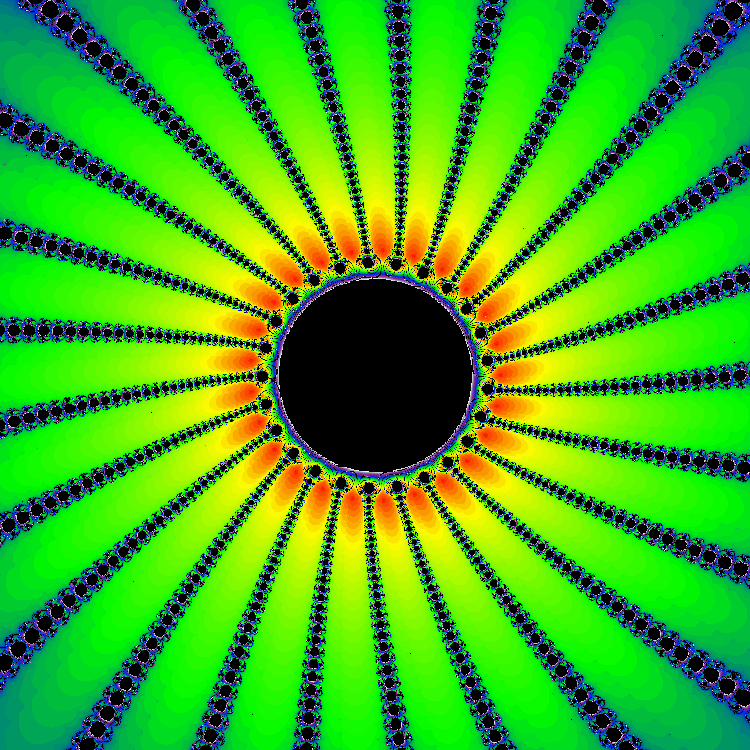};
    \end{axis}
  \end{tikzpicture}
    }
      \caption{\small Dynamical planes of $O_{n, \alpha}(z)$ for  $n=25$ and different values of $\alpha$.\label{fig:dynam25}}
    \end{figure}

In Figure~\ref{fig:dynam3}, Figure~\ref{fig:dynam10} and Figure~\ref{fig:dynam25} we show the dynamical planes of the operator $O_{n,\alpha}$ for several of the most relevant parameters studied along the paper for $n=3$, $n=10$ and $n=25$, respectively. The parameters analysed are the following.
 Subfigure (a) corresponds to  $\alpha=\frac{2n-1}{3n-3}$, which provides the method with order of convergence 4 to the roots (see Proposition~\ref{grado4}). Subfigures (b), (c), and (d) correspond to $\alpha=1/2$ (Halley's method), $\alpha=0$ (Chebyshev's method) and $\alpha=1$ (Super-Halley method), respectively. Subfigure (e) corresponds to $\alpha =\frac{1+ \sqrt{2\left(1-n\right) }}{n-1}$, which is a bifurcation parameter for which the free critical points are mapped to $z=0$ (see Lemma~\ref{lemmaprecrit}). Finally, subfigure (f) corresponds to $\alpha=4i$ which is a parameter in the unbounded component of the complement of the Collar for which the Julia set is disconnected.

A priori, the parameter that should provide a better dynamical behaviour is $\alpha=\frac{2n-1}{3n-3}$ (subfigure (a)), since the corresponding operators have order of convergence 4 to the roots. We observe that, near the roots, it is the parameter with a fastest rate of convergence (it has the most intense red near the roots). However, the global dynamics are  complex. Indeed, for $n=10$ and $n=25$ we observe how the Julia set becomes complicated and there appear black disks of initial conditions which do not converge to the roots after 75 iterates. We want to point out that if we increase the number of iterates we can make disappear this disk for $n=10$. Nevertheless, for $n=25$ the disk decreases very slowly when we increase the number of iterates. This is a kind of pathological behaviour which is not related to stable behaviour different from the roots and can lead to big sets of initial conditions which do not converge to the roots.
The parameters $\alpha=0$ (Chebyshev's method, subfigure (c)),  $\alpha =\frac{1+ \sqrt{2\left(1-n\right) }}{n-1}$ (subfigure (e)) and $\alpha=4i$ (subfigure (f)) present a similar dynamical behaviour. From them, the one which presents a larger black disk around $z=0$ is $\alpha=0$. This indicates that Chebyshev's method may not be an outstanding root finding algorithm to apply to the family $z^n+c$  when $n$ is big. On the other hand, the parameter  $\alpha =\frac{1+ \sqrt{2\left(1-n\right) }}{n-1}$, which is a bifurcation parameter in the Collar, presents a better dynamical behaviour than the parameter  $\alpha=4i$, which is a parameter corresponding to an operator with a disconnected Julia set. Indeed, the speed of convergence to the roots of the operators corresponding to Chebyshev's method and $\alpha =\frac{1+ \sqrt{2\left(1-n\right) }}{n-1}$ is much faster than the one of $\alpha=4i$. We can conclude that the black regions which appear  around bifurcation parameters in the parameter plane (see Figure~\ref{fig:paramgaton} and Figure~\ref{fig:paramzoombif}) are not necessarily related to particularly  bad dynamical behaviour and they are preferable than the parameters corresponding to disconnected Julia set.

 The parameter $\alpha=1$ is particularly interesting. For $n=2$ the corresponding operator has order of convergence 4 to the roots. It still presents very good dynamical behaviour for $n=3$ (see Figure~\ref{fig:dynam3} (d)). However, when we increase $n$ its dynamics get worse very fast. For $n=10$ the corresponding Julia is already quite complex (see Figure~\ref{fig:dynam10} (d)), even if there is no pathological black disk around $z=0$. Notice that for $n=10$ this parameter still belongs to the same hyperbolic component than $\alpha=\frac{2n-1}{3n-3}$. Nevertheless, it is very close to the boundary of the hyperbolic component (see Figure~\ref{fig:paramzoomsing} (a)). For $n=25$, the parameter $\alpha=1$ falls into a cascade of bifurcations (see Figure~\ref{fig:paramzoomsing} (b)). The corresponding dynamics are very bad in terms of convergence to the roots. Indeed, the immediate basins of attraction of the roots are very small for $n=25$. We can conclude that parameters within the cascade of bifurcations which appears around  $\alpha=\frac{2n-1}{2n-2}$ (see Figure~\ref{fig:paramzoomsing}) present bad dynamical behaviour. This is not good news since this bifurcation is very close to the parameter $\alpha=\frac{2n-1}{3n-3}$, which corresponds to the method of order of convergence  4. Hence, a small modification on the parameter may lead to very good or very bad dynamics.

Finally, we want discuss about Halley's method ($\alpha=1/2$). Even if $\alpha=1/2$ is a bifurcation parameter, the structure of the bifurcations around it is very simple (see Figure~\ref{fig:paramzoombif}). We can observe how, for big $n$, this parameter is the one which provides a better dynamical behaviour. Indeed, this is the only parameter with no black disk around $z=0$ for $n=25$. Despite that the convergence to the roots is not as fast as for $\alpha=\frac{2n-1}{3n-3}$, the stability of the dynamics makes of Halley's method a better root finding algorithm for the family $z^n+c$.  We conclude that Halley's method is the best member of the Chebyshev-Halley family when applied to $z^n+c$.

\bigskip

\bigskip

\textbf{Acknowledgments:} The first and third authors were supported by the
Spanish project MTM2014-52016-C02-2-P, the Generalitat Valenciana Project
PROMETEO/2016/089 and UJI project P1.1B20115-16. The second author was
supported by the ANR grant Lambda ANR-13-BS01-0002.

\bibliographystyle{plain}
\bibliography{bibliografia}

\end{document}